\pdfoutput=1
\documentclass[preprint,10pt]{elsarticle}

\usepackage{fullpage}

\usepackage[colorlinks=true]{hyperref}

\usepackage{amsmath,amssymb,amsfonts,amsthm}
\usepackage{hhline}
\theoremstyle{definition}

\theoremstyle{lemma}
\newtheorem{lemma}{Lemma}
\newtheorem*{remark}{Remark}
\theoremstyle{theorem}
\newtheorem{theorem}{Theorem}
\theoremstyle{assumption}
\newtheorem{assumption}{Assumption}

\DeclareMathOperator*{\argmin}{\arg\,\min}

\usepackage{algorithm}
\usepackage{algorithmic}

\usepackage[titletoc,toc,title]{appendix}
\usepackage{tikz}
\usepackage{graphicx, color, bm}

\graphicspath{{./figs/}}

\usepackage{subfig}
\usepackage{pgfplots}
\usepackage{pgfplotstable}

\renewcommand{\hat}{\widehat}
\renewcommand{\tilde}{\widetilde}

\newcommand*\diff[1]{\mathop{}\!{\mathrm{d}#1}}
\newcommand{\diag}[1]{{\rm diag}\LRp{#1}}
\newcommand{\td}[2]{\frac{{\rm d}#1}{{\rm d}{\rm #2}}}
\newcommand{\pd}[2]{\frac{\partial#1}{\partial#2}}
\newcommand{\nor}[1]{\left\| #1 \right\|}
\newcommand{\LRp}[1]{\left( #1 \right)}
\newcommand{\LRs}[1]{\left[ #1 \right]}

\newcommand{\LRb}[1]{\left| #1 \right|}
\newcommand{\LRc}[1]{\left\{ #1 \right\}}

\newcommand{\LRl}[1]{\left. #1 \right|}
\newcommand{\pdn}[3]{\frac{\partial^{#3}#1}{\partial#2^{#3}}}

\newcommand{\avg}[1] {\ensuremath{\LRc{\!\{#1\}\!}}}

\newcommand{\bnote}[1]{#1}
\newcommand{\rnote}[1]{#1}

\begin{document}

\begin{frontmatter}
\title{Entropy stable reduced order modeling of nonlinear conservation laws}

\author[rice]{Jesse Chan}
\ead{Jesse.Chan@rice.edu}
\address[rice]{Department of Computational and Applied Mathematics, Rice University, 6100 Main St, Houston, TX, 77005}

\begin{abstract}
Reduced order models of nonlinear conservation laws in fluid dynamics do not typically inherit stability properties of the full order model. We introduce projection-based hyper-reduced models of nonlinear conservation laws which are globally conservative and inherit a semi-discrete entropy inequality independently of the choice of basis and choice of parameters.  
\end{abstract}
\end{frontmatter}



\section{Introduction}

Projection-based model reduction constructs low-dimensional surrogate models for many-query scenarios (e.g., simulations over multiple parameter values) that can be evaluated over a range of parameters at a low \textit{online} cost in exchange for a more expensive \textit{offline} pre-computation step \cite{benner2015survey}.  The development of reduced order models (ROMs) is relatively mature for several classes of problems (e.g., linear time-invariant systems, coercive elliptic PDEs).  However, the construction of robust and stable ROMs for transient and convection-dominated problems remains an active area of research \cite{cagniart2019model}.  

For certain PDEs, ``structure-preserving'' ROMs provide robustness and stability by reproducing energetic properties of the full system at the discrete level.  ROMs which retain either Lagrangian \cite{lall2003structure, carlberg2015preserving} or Hamiltonian structure \cite{benner2012interpolation, gugercin2012structure, peng2016symplectic, chaturantabut2016structure, gong2017structure, afkham2017structure, afkham2018structure} have been constructed by combining Galerkin projection with an appropriate formulation of the full order model, and similar energy-conserving ROMs have been constructed for the incompressible Navier-Stokes equations \cite{farhat2014dimensional, farhat2015structure}.  The construction of structure-preserving ROMs for nonlinear conservation laws in fluid flow, however, remains an open problem.  For example, Galerkin projection yields ROMs which become unstable as the number of reduced basis functions (modes) is increased \cite{bui2007goal, carlberg2013gnat}.  As a result, stabilized discretizations are often employed.  Petrov-Galerkin ROMs, which use an alternative test basis \cite{maday2002blackbox, rozza2007stability, bui2007goal, serre2012reliable, amsallem2012stabilization, rozza2013reduced, ballarin2015supremizer, carlberg2017galerkin}, are a popular alternative, as are additional residual-based stabilization or dissipation terms \cite{wang2012proper, kalashnikova2014stabilization, caiazzo2014numerical, balajewicz2016minimal}.  Such models improve robustness in practice, though they do not provide a theoretical foundation for stability.  

The construction of stable ROMs for the compressible Navier-Stokes equations is further complicated by the non-trivial structure of the equations.  In response, practitioners have developed structure-preserving entropy stable ROMs for simplifications of the underlying PDE (such as the time-dependent \textit{linearized} compressible Navier-Stokes equations \cite{barone2009stable, kalashnikova2010stability, amsallem2012stabilization, kalashnikova2014construction}) and extrapolate such models to the full nonlinear equations.  A promising alternative is to enforce physical conditions such as such as kinetic energy preservation, which are empirically related to stability \cite{maboudi2018conservative}.  These methods significantly improve the robustness of numerical methods in practice, and guarantee a discrete entropy inequality for systems of nonlinear conservation laws which yield stable split formulations \cite{fisher2013discretely}.  

The approach taken in this work differs from the existing literature in the treatment of nonlinear terms.  In \cite{maboudi2018conservative}, equations are rewritten such that the nonlinear terms involve only polynomial (quadratic and cubic) nonlinearities, which can be evaluated exactly using precomputed matrices.  In this work, we approximate nonlinear terms using hyper-reduction techniques based on empirically computed quadrature rules \cite{an2008optimizing, hernandez2017dimensional}.  The hyper-reduction approach incurs approximation error and additional computational cost; however, it is allows for the generalization of the split forms in \cite{maboudi2018conservative} to nonlinearly entropy stable formulations.  We directly construct entropy stable ROMs for nonlinear conservation laws by combining hyper-reduction with a modified Galerkin projection of an appropriate full order model.  The approach taken in this paper combines techniques from entropy stable finite volume schemes \cite{tadmor1987numerical, tadmor2003entropy} and entropy stable summation-by-parts (SBP) discretizations \cite{fisher2013high, carpenter2014entropy, chen2017entropy, crean2018entropy, chan2017discretely, chan2019skew} to produce entropy stable reduced order models.  

\rnote{We note that this work focuses on classical model reduction techniques, which lose effectiveness for general transport-type phenomena \cite{reiss2018shifted, abgrall2018model, rim2018transport, cagniart2019model}.  This is tied to difficulties in approximating convected solutions using a fixed reduced basis.  Despite these challenges, classical approaches are still in model reduction of transport-type equations for specific problem setups \cite{carlberg2013gnat}.  In this work, we restrict ourselves to classical model reduction techniques.  Challenges associated with the low-dimensional approximation of transport-type solutions will be addressed in future work. }

The paper is organized as follows.  Section~\ref{sec:2} introduces a full order model on 1D periodic domains based on entropy stable finite volume schemes.  Section~\ref{sec:3} describes how to construct an entropy stable reduced basis approximation, while Section~\ref{sec:4} discusses entropy stable hyper-reduction techniques to reduce costs associated with the evaluation of nonlinear terms.  Section~\ref{sec:6} describes how to extend the aforementioned approaches to non-periodic boundary conditions, and Section~\ref{sec:7} describes the extension to higher dimensions.  We conclude in Section~\ref{sec:8} with numerical experiments which verifying the presented theoretical results.


\section{The full order model: entropy stable finite volume schemes}
\label{sec:2}

We briefly summarize entropy inequalities associated with systems of nonlinear conservation laws.  Let $\Omega$ denote some domain with boundary $\partial \Omega$.  Nonlinear conservation laws are expressed as a system of partial differential equations (PDEs) 
\begin{equation}
\pd{\bm{u}}{t}  + \sum_{i=1}^d\pd{\bm{f}_i(\bm{u})}{x_i} = 0, \qquad 
S(\bm{u}) \text{ is a convex function}, \qquad
\bm{v}(\bm{u}) = \pd{S}{\bm{u}},
\label{eq:nonlineqs}
\end{equation}
where $\bm{u}\in \mathbb{R}^n$ are the conservative variables, $\bm{f}_i$ are nonlinear fluxes, and $\bm{v}(\bm{u})$ are the \textit{entropy variables}.  By multiplying (\ref{eq:nonlineqs}) by the entropy variables, viscosity solutions \cite{oleinik1957discontinuous, kruvzkov1970first}  of popular fluid systems (e.g., shallow water, compressible Euler and Navier-Stokes \cite{hughes1986new, chen2017entropy}) can be shown to satisfy 
\begin{equation}
\int_{\Omega}\pd{S(\bm{u})}{t}\diff{\bm{x}} + \sum_{i=1}^d \int_{\partial \Omega} \LRp{\bm{v}^T\bm{f}_i(\bm{u}) - \psi_i(\bm{u})}n_i \leq 0\label{eq:entropyineq},
\end{equation}
where $n_i$ denotes the $i$th component of the outward normal vector, \bnote{and $\psi_i(\bm{u})$ denotes the entropy potential for the $i$th coordinate}.  
The entropy inequality (\ref{eq:entropyineq}) is the analogue of energy stability for nonlinear conservation laws \cite{mock1980systems, harten1983symmetric}.  
However, due to the use of inexact treatment of nonlinear terms (e.g., from collocation approximations or inexact quadrature), most numerical methods for nonlinear conservation laws do not satisfy a discrete analogue of this stability condition.  

The reduced order models in this work are constructed using full order models based on entropy stable finite volume methods \cite{tadmor1987numerical}, which reproduce a discrete version of the continuous entropy inequality (\ref{eq:entropyineq}).  For simplicity, we illustrate the construction of full and reduced order models on a periodic 1D domain.  We note that, while the reduced order model will be constructed from an entropy stable scheme, the solution snapshots do not need to be generated by an entropy stable full order model.  Throughout this work, we will assume that the both full and reduced models yield ``physically relevant'' solutions, such that the entropy is a convex function.  For example, for the compressible Euler and Navier-Stokes equations, we will assume that the thermodynamic variables (density and energy/temperature/pressure) are positive.  Guaranteeing both positivity and accuracy of numerical solutions for general discretizations remains challenging \cite{guermond2016invariant, guermond2019invariant}, and will be explored in future work.

A key ingredient of the full and reduced order models considered in this work are \textit{entropy conservative} finite volume numerical fluxes.  Let $\bm{u}_L, \bm{u}_R$ denote left and right states.  Then, a two-point numerical flux $\bm{f}_S(\bm{u}_L, \bm{u}_R)$ is entropy conservative if it satisfies the following three conditions
\begin{gather}
\bm{f}_S(\bm{u},\bm{u}) = \bm{f}(\bm{u}), \qquad \text{(consistency)} \label{eq:esflux}\\
\bm{f}_S(\bm{u}_L,\bm{u}_R) = \bm{f}_S(\bm{u}_R,\bm{u}_R), \qquad \text{(symmetry)} \nonumber\\
\LRp{\bm{v}_L-\bm{v}_R}^T\bm{f}_S(\bm{u}_L,\bm{u}_R) = \psi(\bm{u}_L) - \psi(\bm{u}_R), \qquad \text{(entropy conservation)}\nonumber,
\end{gather}
These fluxes are used to construct entropy conservative and entropy stable finite volume schemes \cite{tadmor1987numerical, chandrashekar2013kinetic, ray2016entropy}.  Let the domain be decomposed into $K$ cells of size $\Delta x$, and let ${\LRp{\bm{u}_h}_i}$ denote the vector containing mean values of the vector of conservation variables over the $i$th cell.  An entropy conservative finite volume method results from discretizing the integral form of the conservation law as follows
\begin{align}
&\td{{\LRp{\bm{u}_h}_1}}{t} + \frac{\bm{f}_S({\LRp{\bm{u}_h}_2},\LRp{\bm{u}_h}_1)-\bm{f}_S({\LRp{\bm{u}_h}_1},{\LRp{\bm{u}_h}_K})}{\Delta x} = \bm{0}, \label{eq:nonlineqs1d}\\
&\td{{\LRp{\bm{u}_h}_i}}{t} + \frac{\bm{f}_S({\LRp{\bm{u}_h}_{i+1}},\LRp{\bm{u}_h}_i)-\bm{f}_S({\LRp{\bm{u}_h}_i},{\LRp{\bm{u}_h}_i-1})}{\Delta x} = \bm{0}, \qquad 1 < i < K \nonumber\\
&\td{{\LRp{\bm{u}_h}_K}}{t} + \frac{\bm{f}_S({\LRp{\bm{u}_h}_1},\LRp{\bm{u}_h}_K)-\bm{f}_S({\LRp{\bm{u}_h}_K},{\LRp{\bm{u}_h}_K-1})}{\Delta x} = \bm{0},\nonumber
\end{align}
where periodicity is imposed through the equations for ${\LRp{\bm{u}_h}_1}, {\LRp{\bm{u}_h}_K}$.  

We first rewrite the system (\ref{eq:nonlineqs1d}) in a matrix form which is more amenable to model reduction.  We define the skew-symmetric differentiation matrix $\bm{Q}$ and flux matrix $\bm{F}$ such that
\begin{equation}
\bm{Q} = \frac{1}{2}\begin{bmatrix}
0 & 1 & &\ldots & -1\\
-1 & 0 & 1 &&  \\
& -1 & 0 & 1 &  \\
 & & & \ddots &  \\
1 & &\ldots  & -1 & 0
\end{bmatrix}, \qquad \bm{F}_{ij} = \bm{f}_{S}\LRp{\LRp{\bm{u}_h}_i, \LRp{\bm{u}_h}_j}.
\label{eq:Qmat}
\end{equation}
\bnote{Let $\bm{1}$ denote the vector of all ones.}  Note that $\bm{Q}\bm{1} = \bm{0}$, that the matrix $\bm{F}$ is symmetric (due to symmetry of $\bm{f}_S$), and that the diagonal of $\bm{F}$ is equal to $\bm{f}(\bm{u})$ due to flux consistency.  For a scalar nonlinear conservation law, the matrix-based formulation of (\ref{eq:nonlineqs1d}) is then equivalent to
\begin{align}
\Delta x \td{\bm{u}_h}{t} + 2\LRp{{\bm{Q}}\circ \bm{F}}\bm{1} = \bm{0}.
\label{eq:es}
\end{align}
where $\circ$ denotes the Hadamard product, and each entry of $\bm{u}_h$ corresponds to a point $\bm{x}_i$ in the domain.  In this setting, $\bm{Q}\circ\bm{F}$ extracts and takes linear combinations of nonlinear flux interactions between different nodal values of $(\bm{u}_h)_i$ and $(\bm{u}_h)_j$.  \rnote{We note that this reformulation using the Hadamard product is non-standard within the finite volume literature, but is more common in the SBP finite difference literature}.  The discretization of the nonlinear flux term using $\LRp{\bm{Q}\circ \bm{F}}\bm{1}$ is commonly referred to as \textit{flux differencing} \cite{carpenter2014entropy, gassner2016split, chen2017entropy, crean2018entropy, chan2017discretely}.\footnote{The factor of 1/2 present in the definition of $\bm{Q}$ is to make (\ref{eq:es}) consistent with the entropy stable SBP literature.  The factor of 2 in (\ref{eq:es}) can also be derived from the chain rule \cite{chen2017entropy, crean2018entropy}.}  

For a system of $n$ nonlinear conservation laws, all matrices are treated in a Kronecker product fashion.  Let the vector $\bm{u}_h$ now correspond to the concatenated vector of components of the solution $\bm{u}_h = [\bm{u}_h^1, \ldots, \bm{u}_h^n]^T$.  Then, the matrix formulation of (\ref{eq:nonlineqs1d}) is given by
\begin{gather*}
\Delta x\td{\bm{u}_h}{t} + \LRp{\LRp{\bm{I}_{n \times n} \otimes \bm{Q} }\circ \bm{F}}\bm{1} = \bm{0},\\
\bm{F} = \begin{bmatrix}
\bm{F}_{1} & &\\
& \ddots &\\
&& \bm{F}_{n} \\
\end{bmatrix}, \qquad 
\LRp{\bm{F}_{i}}_{jk} = \bm{f}_{i,S}\LRp{\LRp{\bm{u}_h}_j, \LRp{\bm{u}_h}_k}, \qquad \LRp{\bm{u}_h}_j = \begin{bmatrix}\LRp{\bm{u}^1_h}_j & \LRp{\bm{u}^2_h}_j & \ldots & \LRp{\bm{u}^n_h}_j\end{bmatrix}^T.
\label{eq:essys}
\end{gather*}
Here, $\bm{F}$ is a block diagonal matrix where the $i$th block corresponds to the evaluation of the $i$th component of the numerical flux $\bm{f}_{i,S}(\bm{u}_L,\bm{u}_R)$.  From this point on, we drop explicit references to components and the Kronecker product to simplify notation, so that (\ref{eq:es}) applies to both scalar equations and systems.

We next show that the full order model satisfies a discrete conservation of entropy.  
\begin{theorem}
\label{thm:ecfom}
Let $\bm{u}_h(t)$ be a solution of (\ref{eq:essys}) for which the entropy $S(\bm{u}_h)$ is convex at each time $t$, \bnote{and let the flux $\bm{f}_S(\bm{u}_L,\bm{u}_R)$ be entropy conservative as defined by (\ref{eq:esflux})}.  Then, $\bm{u}_h$ satisfies the semi-discrete conservation of entropy 
\[
\Delta x \bm{1}^T\td{S(\bm{u}_h)}{t} = 0.
\]
\end{theorem}
\begin{proof}
The proof can be found in the literature \cite{tadmor1987numerical, tadmor2003entropy, tadmor2016entropy}.  Inspired by entropy stable summation by parts (SBP) schemes \cite{carpenter2014entropy, gassner2016split, chen2017entropy, crean2018entropy, chan2017discretely}, we present an alternative proof which relies only on matrix properties of $\bm{Q}$.  We test (\ref{eq:es}) with the vector of entropy variables $\bm{v}_h = \bm{v}(\bm{u}_h)$
\[
\Delta x \bm{v}_h^T\td{\bm{u}_h}{t} + \bm{v}_h^T2\LRp{{\bm{Q}}\circ \bm{F}}\bm{1} = 0.
\]
Assuming continuity in time and using the definition of the entropy variables in (\ref{eq:nonlineqs}), $\bm{v}_h^T\td{\bm{u}_h}{t}$ simplifies to
\begin{align*}
{\td{S(\bm{u}_h)}{\bm{u}}}^T\td{\bm{u}_h}{t} = \sum_j \LRp{\td{S(\bm{u}_h)}{\bm{u}}}_j^T\td{\LRp{\bm{u}_h}_j}{t} =  \sum_j \td{S(\LRp{\bm{u}_h}_j)}{t}  =  \bm{1}^T\td{S(\bm{u}_h)}{t}.
\end{align*}
Using skew-symmetry of $\bm{Q}$, the flux term $\bm{v}_h^T2\LRp{{\bm{Q}}\circ \bm{F}}\bm{1}$ yields
\begin{align*}
\sum_{ij} \LRp{\bm{v}_h}_i^T \bnote{2}\bm{Q}_{ij} \bm{f}_{S}\LRp{\LRp{\bm{u}_h}_i, \LRp{\bm{u}_h}_j} = \sum_{ij} \LRp{\bm{Q}_{ij}-\bm{Q}_{ji}} \LRp{\bm{v}_h}_i^T  \bm{f}_{S}\LRp{\LRp{\bm{u}_h}_i, \LRp{\bm{u}_h}_j}.
\end{align*}
Rearranging indices and using the symmetry of $\bm{f}_S(\bm{u}_L,\bm{u}_R) = \bm{f}_S(\bm{u}_R,\bm{u}_L)$ exposes the entropy conservation condition in the sum
\begin{align*}
&\sum_{ij} \LRp{\bm{Q}_{ij}-\bm{Q}_{ji}} \LRp{\bm{v}_h}_i^T \bm{f}_{S}\LRp{\LRp{\bm{u}_h}_i, \LRp{\bm{u}_h}_j} \\
=&\sum_{ij} \bm{Q}_{ij} \LRp{\LRp{\bm{v}_h}_i-\LRp{\bm{v}_h}_j}^T  \bm{f}_{S}\LRp{\LRp{\bm{u}_h}_i, \LRp{\bm{u}_h}_j}\\
=&\sum_{ij} \bm{Q}_{ij} \LRp{\psi\LRp{\LRp{\bm{u}_h}_i} - \psi(\LRp{\bm{u}_h}_j)} = \bm{\psi}^T\bm{Q}\bm{1}-\bm{1}^T\bm{Q}\bm{\psi} = 0,
\end{align*}
where we have used that $\bm{Q}=-\bm{Q}^T$ and $\bm{Q}\bm{1} = \bm{0}$ in the final step.

\end{proof}

\subsection{Viscosity and entropy dissipation} 
\label{sec:entropydissipation}
Theorem~\ref{thm:ecfom} shows that the formulation (\ref{eq:essys}) preserves a semi-discrete conservation of entropy.  However, in the presence of shock discontinuities, entropy should be dissipated instead of conserved.  To mimic this at the semi-discrete level, we add entropy dissipation to (\ref{eq:es}) through appropriate viscosity terms.  For example, it was shown in \cite{tadmor2006entropy} that a centered approximation of the Navier-Stokes viscosity is entropy dissipative.  In this work, we apply a simple Laplacian artificial viscosity to each component of the solution \cite{upperman2019entropy}
\begin{equation}
\Delta x \td{\bm{u}_h}{t} + 2\LRp{{\bm{Q}}\circ \bm{F}}\bm{1} + \epsilon \bm{K}\bm{u}_h= 0, \qquad \bm{K} = \frac{1}{\bnote{\Delta x}} \begin{bmatrix}
1 & -1 & & \\
-1 & 2 & \ddots &\\
 & \ddots & \ddots & -1\\
 &  & -1& 1\\
\end{bmatrix}
\label{eq:esvisc}
\end{equation}
where $\epsilon$ is the visosity coefficient.  This choice is intended to simplify the presentation of entropy stable treatments of diffusion terms.  Future work will analyze both physically relevant viscosities and more nuanced artificial dissipation mechanisms.  

\begin{remark}
\bnote{The choice of Neumann boundary conditions for the Laplacian matrix $\bm{K}$ is arbitrary, and $\bm{K}$ can be replaced with the periodic Laplacian.  Future works will investigate different viscous boundary conditions (such as solid wall conditions) and their impacts on reduced order models.}
\end{remark}

If $\bm{v}_h^T\bm{K}\bm{u}_h \geq 0$, then the solution satisfies a discrete dissipation of entropy
\begin{equation}
\Delta x \bm{1}^T\td{S(\bm{u}_h)}{t} = -\epsilon\bm{v}_h^T\bm{K}\bm{u}_h \leq 0.
\label{eq:entropybalancefom}
\end{equation}
The proof is similar to the one given in \cite{tadmor2006entropy}.  We rewrite the viscous term as
\[
\bm{v}_h^T\bm{K}\bm{u}_h = \bm{v}_h^T \frac{1}{\Delta x^2}\begin{bmatrix}
\LRp{\bm{u}_h}_1 - \LRp{\bm{u}_h}_2\\ 
\vdots\\
\LRp{\LRp{\bm{u}_h}_i - \LRp{\bm{u}_h}_{i-1}} - \LRp{\LRp{\bm{u}_h}_{i+1} - \LRp{\bm{u}_h}_{i}}\\
\vdots\\
\LRp{\bm{u}_h}_K- \LRp{\bm{u}_h}_{K-1} 
\end{bmatrix}.
\]
Since $\bm{v}_h$ and $\bm{u}_h$ are related through a differentiable invertible mapping, the mean value theorem implies that
\[
\LRp{\bm{u}_h}_i - \LRp{\bm{u}_h}_{i-1} = \LRp{\pd{\bm{u}}{\bm{v}}}_{i,i-1} \LRp{\LRp{\bm{v}_h}_i - \LRp{\bm{v}_h}_{i-1}}
\]
where $\LRp{\pd{\bm{u}}{\bm{v}}}_{i,i-1}$ is the Jacobian matrix evaluated at some state between $(\bm{u}_h)_i$ and $(\bm{u}_h)_{\bnote{i-1}}$.\footnote{\rnote{We note that this ``average'' Jacobian is evaluated componentwise via a path integral \cite{tadmor2003entropy}.}}  Since the Jacobian is related to the Hessian of the entropy through $\pdn{S(\bm{u})}{\bm{u}}{2} = \pd{\bm{v}}{\bm{u}} = \LRp{\pd{\bm{u}}{\bm{v}}}^{-1}$, it is positive definite so long as $S(\bm{u})$ is convex.  Substituting these expressions into the viscous terms and using positive definiteness of $\pd{\bm{u}}{\bm{v}}$  yields 
\begin{gather}
\bm{v}_h^T\bm{K}\bm{u}_h = 
 \frac{1}{\Delta x^2}\sum_{i=1}^{K-1} \LRp{\pd{\bm{u}}{\bm{v}}}_{i+1,i} \LRb{\LRp{\bm{v}_h}_{i} - \LRp{\bm{v}_h}_{i+1}}^2 \geq 0.
 \label{eq:entropydissfom}
\end{gather}

\section{An entropy conservative reduced order model}
\label{sec:3}
We can now formulate a reduced basis approximation on periodic domains using the entropy stable full order models described in the previous section.  Boundary conditions will be addressed later in Section~\ref{sec:6}.  We also note that this reduced order model is not practical, since the cost of evaluating nonlinear terms scales with the size of the full order model.  Later sections will discuss how to reduce the cost of nonlinear evaluations using appropriate hyper-reduction techniques.

Let $\LRc{{\phi}_j(x)}_{j=1}^{N}$ denote a reduced basis for each component of the solution, which may be generated using (for example) principal orthogonal decomposition (POD) or the reduced basis procedure.  Let $\bm{V}$ denote the generalized Vandermonde matrix whose columns contain evaluations of $\phi_j$ at grid points $x_i$
\[
\bm{V}_{ij} = {\phi}_j(x_i).
\]
\rnote{We assume that the solution is well-approximated in the reduced basis over the entire time-window of the simulation, } and approximate grid values of the solution by $\bm{u}_h = \bm{V}\bm{u}_N$.  Here, $\bm{u}_N$ denote ``modal'' coefficients  of the solution in the reduced basis.  \rnote{Again, we emphasize that solution snapshots for transport-type equations may not be well-approximated by a low-dimensional reduced basis.  Future work will attempt to combine techniques introduced in this work with methods to address this approximation issue \cite{reiss2018shifted, rim2018transport, cagniart2019model}.}

Ignoring viscosity terms for now, plugging this expression into (\ref{eq:es}) and enforcing that the residual is orthogonal to all columns of $\bm{V}$ (Galerkin projection) yields a reduced system
\begin{align*}
\Delta x\bm{V}^T\bm{V}\td{\bm{u}_N}{t} + 2\bm{V}^T\LRp{\bm{Q}\circ \bm{F}}\bm{1} = 0.
\end{align*}
We will show that, while this formulation is not entropy conservative, a slight modification recovers semi-discrete entropy conservation.  We motivate this modification by first considering the time derivative term.  Because of the Galerkin projection, we can no longer directly test with the vector of entropy variables, which may not lie  in the span of the reduced basis functions.  However, we can test with an appropriate \textit{projection} of the entropy variables.  Let $\bm{v}_N$ denote coefficients of the projection of the entropy variables
\[
\bm{v}_N = \LRp{\bm{V}^T\bm{V}}^{-1}\bm{V}^T \bm{v}\LRp{\bm{V}\bm{u}_N} = \bm{V}^{\dagger} \bm{v}\LRp{\bm{V}\bm{u}_N},
\]
where $\bm{V}^{\dagger}$ is the pseudoinverse of $\bm{V}$.  Then, we can recover the time derivative of the discrete entropy by testing with the coefficients $\bm{v}_N$
\begin{align}
\bm{v}_N^T\bm{V}^T\bm{V}\td{\bm{u}_N}{t} &= \bm{v}\LRp{\bm{V}\bm{u}_N}^T \bm{V}\LRp{\bm{V}^T\bm{V}}^{-1}\bm{V}^T\bm{V} \td{\bm{u}_N}{t} \label{eq:esromtime}\\
&= 
 \bm{v}\LRp{\bm{V}\bm{u}_N}^T \td{\LRp{\bm{V}\bm{u}_N}}{t} = \bm{1}^T\td{S(\bm{V}\bm{u}_N)}{t}.\nonumber
\end{align}
The remainder of the proof requires showing that 
\[
\bm{v}_N^T\bm{V}^T\LRp{\bm{Q}\circ \bm{F}}\bm{1} = \tilde{\bm{v}}^T\LRp{\bm{Q}\circ \bm{F}}\bm{1}  = 0, \qquad \tilde{\bm{v}} = \bm{V}\bm{v}_N = \bm{V}\bm{V}^{\dagger} \bm{v}\LRp{\bm{V}\bm{u}_N},
\]
where we have introduced the grid values of the projected entropy variables as $\tilde{\bm{v}}$.
We can repeat the steps of the proof of Theorem~\ref{thm:ecfom} up to the point when we invoke the entropy conservation condition of (\ref{eq:esflux})
\begin{align*}
\tilde{\bm{v}}^T\LRp{{\bm{Q}}\circ \bm{F}}\bm{1} &= \frac{1}{2}\sum_{ij} \bm{Q}_{ij} \LRp{\tilde{\bm{v}}_i-\tilde{\bm{v}}_j}^T \bm{f}_{S}\LRp{\LRp{\bm{u}_h}_i, \LRp{\bm{u}_h}_j}\\
&\neq \frac{1}{2}\sum_{ij} \bm{Q}_{ij} \LRp{ \psi((\bm{u}_h)_i)- \psi((\bm{u}_h)_j)}.
\end{align*}
This is not possible due to the fact that the projected entropy variables $\tilde{\bm{v}}$ are no longer mappings of the grid values of the conservative variables $\bm{u}_h = \bm{V}\bm{u}_N$.  To remedy this, we follow approaches taken in \cite{parsani2016entropy, chan2017discretely} and replace the grid values of ${\bm{u}_h}$ used to evaluate the flux matrix $\bm{F}$ with values of the \textit{entropy-projected conservative variables} $\tilde{\bm{u}}$
\[
\tilde{\bm{u}} = \bm{u}\LRp{\bm{V}\bm{V}^{\dagger}\bm{v}\LRp{\bm{V}\bm{u}_N}} = \bm{u}\LRp{\tilde{\bm{v}}}.
\]
The entropy projected conservative variables $\tilde{\bm{u}}$ are thus mappings of the projected entropy variables.  We note that $\bm{V}\bm{V}^{\dagger}$ can also be interpreted as a discrete approximation of the $L^2$ projection operator onto the span of the reduced basis $\phi_j$.  One can then construct a semi-discretely entropy conservative reduced model.  
\begin{theorem} 
Let the coefficients $\bm{u}_N$ solve
\begin{gather}
\Delta x\bm{V}^T\bm{V}\td{\bm{u}_N}{t} + 2\bm{V}^T\LRp{\bm{Q}\circ \bm{F}}\bm{1} = \bm{0}, \label{eq:esrom}\\
\LRp{\bm{F}}_{ij} = \bm{f}_S\LRp{\tilde{\bm{u}}_i,\tilde{\bm{u}}_j}, \qquad \tilde{\bm{u}} = \bm{u}\LRp{\bm{V}\bm{V}^{\dagger}\bm{v}\LRp{\bm{V}\bm{u}_N}}.\nonumber
\end{gather}
Then, the solution satisfies the following semi-discrete conservation of entropy
\[
\Delta x\bm{1}^T\td{S(\bm{V}\bm{u}_N)}{t} = 0.
\]
Additionally, if $\bm{1}$ lies within the range of the reduced basis matrix $\bm{V}$ (e.g., $1$ is in the span of the reduced basis functions $\phi_1, \ldots, \phi_N$), solutions of (\ref{eq:esrom}) conserve global averages of the conservative variables.
\label{thm:esrom}
\end{theorem}
\begin{proof}
The conservation of entropy follows from testing with $\bm{v}_N$ and applying (\ref{eq:esromtime}).  The remaining steps are identical to those of Theorem~\ref{thm:ecfom}.  The global conservation results from testing with $\bm{1}$.  If $\bm{1}$ is in the range of $\bm{V}$, then the exist coefficients in the reduced basis matrix $\bm{e}$ such that $\bm{V}\bm{e}=\bm{1}$.  Then, 
\begin{align*}
\Delta x\bm{e}^T\bm{V}^T\bm{V}\td{\bm{u}_N}{t} + 2\bm{e}^T\bm{V}^T\LRp{\bm{Q}\circ \bm{F}}\bm{1}  = \Delta x\bm{1}^T\td{\LRp{\bm{V}\bm{u}_N}}{t} + 2\bm{1}^T\LRp{\bm{Q}\circ \bm{F}}\bm{1} = 0.
\end{align*}
Since $\bm{Q}$ is skew-symmetric and $\bm{F}$ is symmetric, $\bm{Q} \circ \bm{F}$ is skew-symmetric and $\bm{1}^T\LRp{\bm{Q}\circ \bm{F}}\bm{1} = \bm{0}$, 
\[
\Delta x\bm{1}^T\td{\LRp{\bm{V}\bm{u}_N}}{t} = 0,
\]
which is a discrete statement of conservation.
\end{proof}

\begin{remark}
In order for the formulation (\ref{eq:esrom}) to remain accurate, the entropy projected conservative variables must accurately approximate the conservative variables.  This requires that the reduced basis accurately approximates the entropy variables.  This can be taken into account, for example, by computing the POD basis vectors from snapshots of both the conservative and entropy variables.
\end{remark}

\section{Entropy conservative hyper-reduction}
\label{sec:4}
While the Galerkin projected formulation (\ref{eq:esrom}) is entropy stable, the computational cost involved in solving the semi-discrete system scales with the size of the full order model rather than the dimension of the reduced basis.  For example, explicit time-stepping methods require the evaluation of the nonlinear term $\bm{V}^T\LRp{\bm{Q} \circ \bm{F}}\bm{1}$.  Since the matrices $\bm{Q}$ and $\bm{F}$ are the original full order model matrices, the cost of the reduced system is not lower than the cost of the full order model.  

To reduce the cost of evaluating nonlinear terms, we introduce a second \textit{hyper-reduction} step \cite{ryckelynck2009hyper}.  Here, terms involving vectors of nonlinear function evaluations are approximated by terms involving nonlinear function evaluations at a subset of points \cite{barrault2004empirical, bui2004aerodynamic, chaturantabut2010nonlinear, farhat2015structure, drmac2016new, hernandez2017dimensional, yano2019lp}.  In this work, we utilize a sampling and weighting strategy \cite{farhat2015structure, hernandez2017dimensional, yano2019lp} which enables proofs of discrete entropy stability.  These hyper-reduction techniques can be interpreted as reduced quadratures, and produce approximations of the form
\[
\bm{V}^Tg(\bm{V}\bm{u}_N) \approx \bm{V}\LRp{\mathcal{I},:}^T \bm{W} g\LRp{\bm{V}\LRp{\mathcal{I},:}\bnote{\bm{u}_N}}.
\]
Here, $g(\bm{u})$ denotes a nonlinear function, $\mathcal{I}$ denotes a subset of $N_s$ row indices corresponding to sampled points, $\bm{V}\LRp{\mathcal{I},:}$ denotes the sub-matrix consisting of the $N_s$ sampled rows of $\bm{V}$, and $\bm{W} = {\rm diag}(\bm{w})$ is a $N_s\times N_s$ diagonal matrix of positive weights.  We will describe algorithms for computing hyper-reduced points and weights in Section~\ref{sec:hyperreducalgo}.

\rnote{We briefly outline our approach to hyper-reduction.  Rather than directly apply hyper-reduction on the nonlinear vector, we perform a matrix-based hyper-reduction which respects the matrix structure of the nonlinear term $\bm{V}^T\LRp{\bm{Q}\circ \bm{F}}$.  Specifically, we construct a smaller hyper-reduced matrix $\bm{Q}_s$ and approximate the term $\bm{V}^T\LRp{\bm{Q}\circ \bm{F}}$ with $\bm{V}\LRp{\mathcal{I},:}^T\bm{W}\LRp{\bm{Q}_s\circ \bm{F}_s}$, where $\bm{F}_s$ is a smaller hyper-reduced matrix containing flux evaluations between solution states at different hyper-reduced points.  Recall that the proof of entropy conservation for the full order model in Theorem~\ref{thm:ecfom} required only that the matrix $\bm{Q}$ be skew symmetric and have zero row sum. } If the hyper-reduced matrix $\bm{Q}_s$ also satisfies those properties
\[
\bm{Q}_s = -\bm{Q}_s^T, \qquad \bm{Q}_s\bm{1} = 0
\]
then we will be able to extend the proof of entropy conservation in Theorem~\ref{thm:esrom} to the hyper-reduced model.  

Unfortunately, common hyper-reduction techniques \rnote{applied naively to the matrix $\bm{Q}$} preserve either skew-symmetry or zero row sums, but not both.  For example, $\bm{Q}$ can be decomposed into the sum of local skew-symmetric matrices, and the hyper-reduction techniques of \cite{farhat2015structure, yano2019discontinuous} approximate $\bm{Q}$ using a sparse linear combination of those local matrices.  This hyper-reduction technique preserves skew-symmetry, but does not necessarily preserve the zero row sum property.  Similarly, techniques such as gappy POD or empirical interpolation \cite{everson1995karhunen, barrault2004empirical, chaturantabut2010nonlinear} may preserve the zero row sum property but not skew-symmetry.  \rnote{The next section describes a two-step hyper-reduction which retains both skew-symmetry and zero row sums of $\bm{Q}_s$ (though at the cost of sparsity).}

\subsection{Two-step hyper-reduction: compress and project}

We apply a two-step hyper-reduction procedure which preserves both the zero row sums and skew-symmetry of the hyper-reduced matrix $\bm{Q}_s$.  Rather than directly hyper-reducing the full order matrix $\bm{Q}$, we first construct a compressed ``modal'' matrix, then combine this with a projection operator based on the hyper-reduced points.  In the first step, we construct a compressed intermediate operator by combining  Galerkin projection with the ``expanded basis'' approach of \cite{hernandez2017dimensional}.  Let $\bm{V}_t$ denote a \textit{test} basis through which to extract the action of $\bm{Q}$.  \bnote{In this work, we assume that the span of the test basis includes the reduced basis, e.g., $\mathcal{R}(\bm{V})\subset \mathcal{R}(\bm{V}_t)$}.  We define the intermediate reduced operator 
\begin{equation}
\hat{\bm{Q}}_{t} = \bm{V}_{t}^T\bm{Q}\bm{V}_{t}.
\label{eq:compressQ}
\end{equation}
Let $\bm{c}$ be some vector in the span of the test basis such that $\bm{c}=\bm{V}_t \hat{\bm{c}}$. Solving the system $\bm{V}_t^T\bm{V}_t \hat{\bm{d}} = \hat{\bm{Q}}_t\hat{\bm{c}}$ yields that
\[
\bm{V}_t^T\bm{V}_t \hat{\bm{d}} = \bm{V}_t^T\bm{Q}\bm{V}_t\hat{\bm{c}} = \bm{V}_t^T\bm{Q}\bm{c} \qquad \Longrightarrow\qquad \bm{V}_t^T \LRp{\bm{V}_t\hat{\bm{d}}-\bm{Q}\bm{c}} = \bm{0}.
\]
Here, $\hat{\bm{d}}$ are the coefficients of the orthogonal projection of $\bm{Q}\bm{c}$ onto the test basis.  Note that $\hat{\bm{Q}}_t$ is skew-symmetric by construction.  Moreover, if the vector of all ones lies in the span of the test basis (e.g., $\bm{1} = \bm{V}_t\bm{e}$ for some coefficient vector $\bm{e}$), then 
\[
\hat{\bm{Q}}_t\bm{e} = \bm{0},
\]
since $\hat{\bm{Q}}_t$ exactly recovers the action of $\bm{Q}\bm{1} = \bm{0}$ and $\bm{0}$ lies within the span of any basis.  


Since $\hat{\bm{Q}}_{t}$ acts on coefficients of the test basis (a ``modal'' operator), it cannot be directly applied to evaluate nonlinear functions of the solution.  However, we can construct ``nodal'' operator by combining $\hat{\bm{Q}}_t$ with appropriate mappings from hyper-reduced points to coefficients of the test basis.  We make the following assumptions on the hyper-reduced points to enable the construction of such mappings: 
\begin{assumption}
Let $\mathcal{I}$ denote the index set of hyper-reduced points.  We assume that the hyper-reduced test mass matrix $\bm{V}_t\LRp{\mathcal{I},:}^T\bm{W}\bm{V}_t\LRp{\mathcal{I},:}$ is non-singular.
 \label{ass:quad}
\end{assumption}
Note that if Assumption~\ref{ass:quad} holds, then the test mass matrix $\bm{V}\LRp{\mathcal{I},:}^T\bm{W}\bm{V}\LRp{\mathcal{I},:}$ is also non-singular, since the column space of $\bm{V}$ is contained in the column space of $\bm{V}_t$.  \rnote{We will provide a heuristic algorithm for ensuring that the hyper-reduced test mass matrix is non-singular in Section~\ref{sec:condtest}.}

If Assumption~\ref{ass:quad} holds, we can define the projection matrix $\bm{P}_t$ onto the test basis as follows:
\begin{equation}
\bm{P}_t = \LRp{\bm{V}_t\LRp{\mathcal{I},:}^T\bm{W}\bm{V}_t\LRp{\mathcal{I},:}}^{-1}\bm{V}_t\LRp{\mathcal{I},:}^T\bm{W}.
\label{eq:Pt}
\end{equation}
The projection matrix $\bm{P}_t$ maps values at hyper-reduced points to coefficients in the test basis, and is well-defined under Assumption~\ref{ass:quad}.  The key property of $\bm{P}_t$ which we will utilize is as follows: suppose that $\bm{f} = \bm{V}_t\bnote{\LRp{\mathcal{I},:}}\bm{c}$ for some coefficients $\bm{c}$.  Then,
\begin{equation}
\bm{P}_t \bm{f} = \LRp{\bm{V}_t\LRp{\mathcal{I},:}^T\bm{W}\bm{V}_t\LRp{\mathcal{I},:}}^{-1}\bm{V}_t\LRp{\mathcal{I},:}^T\bm{W} \bm{V}_t\LRp{\mathcal{I},:}\bm{c} = \bm{c}.
\label{eq:preproduce}
\end{equation}
In other words, if a vector lies in the range of $\bm{V}_t\LRp{\mathcal{I},:}$, the projection matrix $\bm{P}_t$ recovers the coefficients exactly.  We also note that if the hyper-reduced set $\mathcal{I}$ is taken to include all the points and $\bm{W}$ is a multiple of the identity (for example, $\bm{W} = \Delta x \bm{I}$), then $\bm{P}_t$ reduces to the pseudo-inverse $\bm{V}_t^{\dagger} = \LRp{\bm{V}_t\bnote{^T}\bm{V}_t}^{-1}\bm{V}_t^T$.  

We can now construct a ``nodal'' differentiation operator which satisfies both skew-symmetry and zero row sum properties: 
\begin{lemma}
Suppose that $\bm{1}$ is exactly representable in the test basis $\bm{V}_t$ and that Assumption~\ref{ass:quad} holds.  Define the matrix $\bm{Q}_t$
\begin{equation}
\bm{Q}_t = \bm{P}_t^T\hat{\bm{Q}}_t\bm{P}_t = \bm{P}_t^T\LRp{\bm{V}_t^T\bm{Q}\bm{V}_t}\bm{P}_t.
\label{eq:nodalQ}
\end{equation}
Then, $\bm{Q}_t$ is skew-symmetric and satisfies $\bm{Q}_t\bm{1} = \bm{0}$.  
\label{thm:Qt}
\end{lemma}
\begin{proof}
Since $\bm{Q}$ is skew-symmetric, the matrix $\bm{Q}_t$ is skew-symmetric by construction as well.  By (\ref{eq:preproduce}), if $\bm{1}$ is exactly representable in the test basis (such that $\bm{1} = \bm{V}_t\LRp{\mathcal{I},:}\bm{e}$ for some coefficient vector $\bm{e}$  ), then  $\bm{P}_t\bm{1} = \bm{e}$.  Then, $\bm{Q}_t\bm{1} = \bm{P}_t^T\hat{\bm{Q}}_t\bm{e} = \bm{0}$, since $\hat{\bm{Q}}_t$ exactly differentiates elements of the test basis.
\end{proof}

\begin{remark}
The projection matrix $\bm{P}_t$ can also be replaced, for example, by the gappy POD projection $\LRp{\bm{V}_t\LRp{\mathcal{I},:}}^{\dagger}$ \cite{willcox2006unsteady,astrid2008missing}.  We do not observe significant differences in numerical results when doing so.  In principle, any matrix operator which maps from nodal values to modal coefficients and reproduces the test basis exactly can be used for $\bm{P}_t$.  In this work, we restrict ourselves to the projection matrix defined by (\ref{eq:Pt}), and will investigate other definitions of $\bm{P}_t$ in future work.
\end{remark}

\subsubsection{Choice of test basis}

We have not yet specified the choice of test basis $\bm{V}_t$.  The span of the test basis should include the reduced basis $\bm{V}$, and Lemma~\ref{thm:Qt} \bnote{assumes} that the span of the test basis should also contain $\bm{1}$.  However, these two conditions are insufficient to ensure accuracy for solutions of nonlinear conservation laws.  Consider, for example, the 1D periodic Burgers' equation with initial condition 
\[
u(x,0) = -\sin(\pi x).
\]
The solution $u(x,t)$ is a decaying stationary shock which is anti-symmetric across the origin at all times $t$.  The POD modes are thus also anti-symmetric across the origin; however, their derivatives $\bm{Q}\bm{V}$ are nearly symmetric across the origin, as shown in Figure~\ref{fig:modeQ}.  Since the entries of $\bm{V}^T\bm{Q}\bm{V}$ are inner products of nearly symmetric and anti-symmetric functions, $\bm{V}^T\bm{Q}\bm{V} \approx \bm{0}$, and the resulting hyper-reduced model is highly inaccurate.
\begin{figure}[!h]
\centering
\subfloat[Solution snapshots]{\includegraphics[width=.32\textwidth]{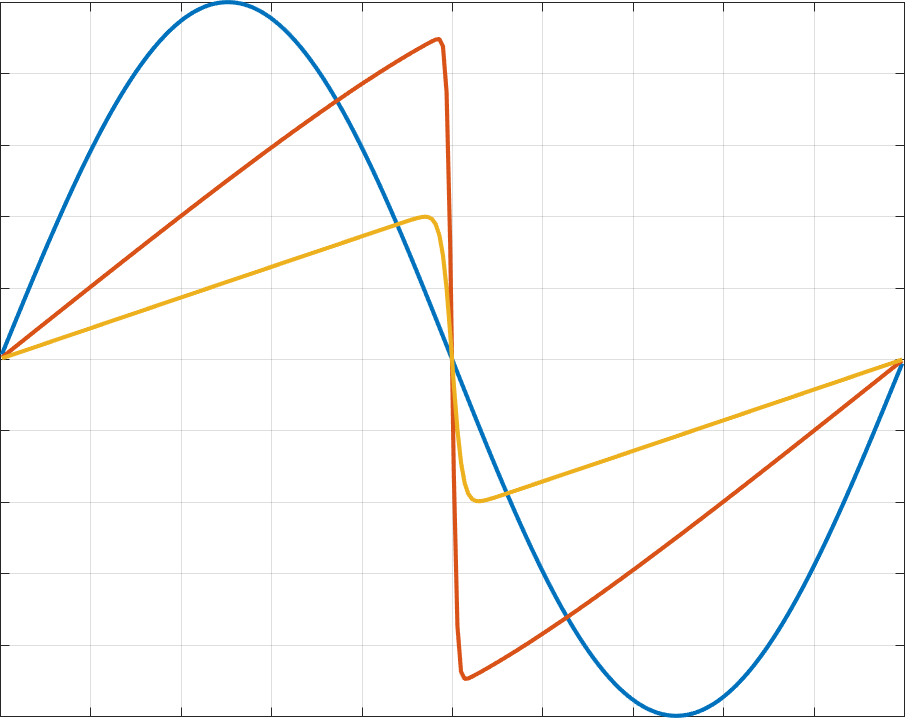}}
\hspace{.1em}
\subfloat[Modes (columns of $\bm{V}$)]{\includegraphics[width=.32\textwidth]{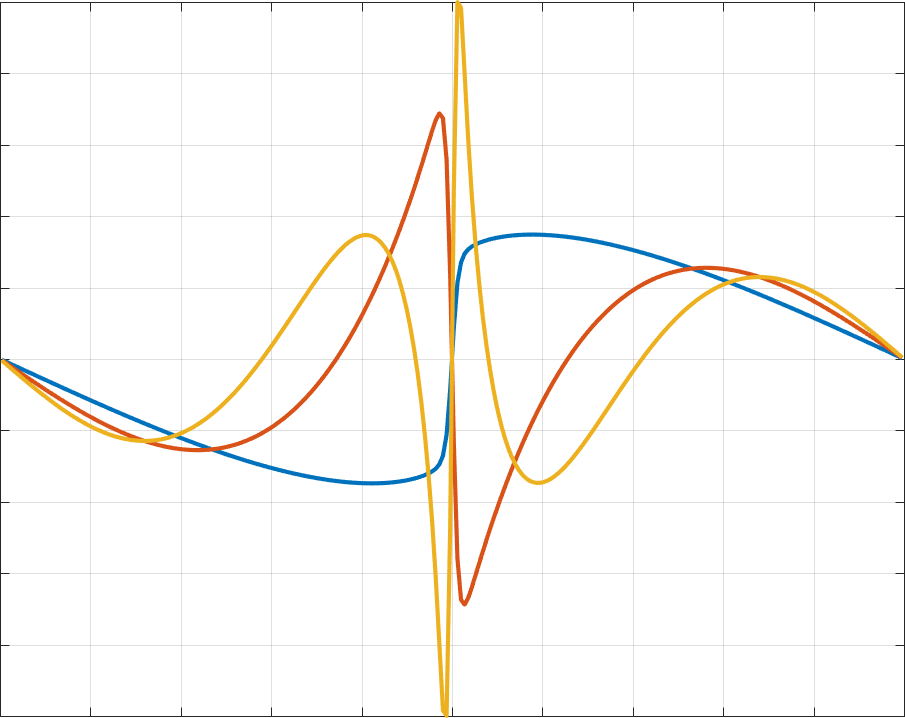}}
\hspace{.1em}
\subfloat[Mode derivatives $\bm{Q}\bm{V}$]{\includegraphics[width=.32\textwidth]{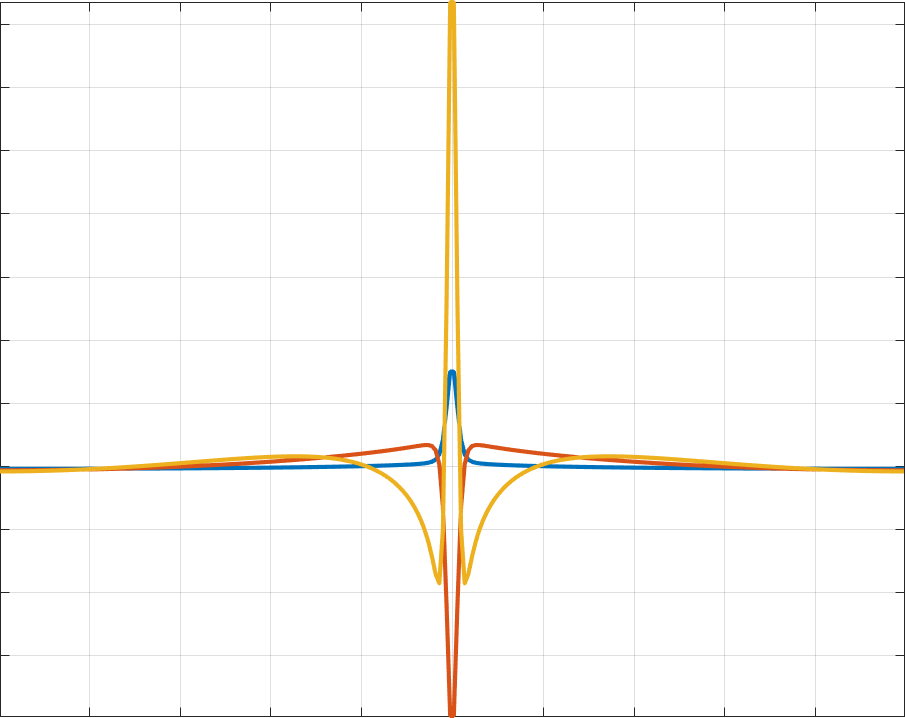}}
\caption{Solution snapshots, three POD modes, and mode derivatives for a shock solution of the periodic Burgers' equation. }
\label{fig:modeQ}
\end{figure}

This example implies that the reduced basis matrix $\bm{V}$ can do a poor job of sampling the range of $\bm{Q}\bm{V}$, and that the test basis should contain additional vectors.  To remedy this issue, we borrow techniques from least squares Petrov-Galerkin ROMs \cite{carlberg2011efficient, carlberg2017galerkin} and enrich the test basis $\bm{V}_t$ with vectors spanning the range of $\bm{Q}\bm{V}$.  The test basis now spans a space which contains $\bm{1}, \bm{V}$, and $\bm{Q}\bm{V}$, and thus has a dimension of at most $2N+1$.  The dimension of $\bm{V}_t$ may also be smaller if the intersection of the range of the reduced basis matrix $\bm{V}$ and the range of $\bm{Q}\bm{V}$ and $\bm{1}$ is non-empty.  

We briefly comment on why this choice of test basis restores accuracy. 
\bnote{The matrix $\bm{Q}_t$ should at minimum reproduce the action of $\bm{Q}$ on the reduced basis $\bm{V}$.  In the absence of hyper-reduction (where $\bm{P}_t = \bm{V}_t^{\dagger}$), we have that
\[
\bm{Q}_t \bm{V} = \LRp{\bm{V}_t\bm{V}_t^{\dagger}}^T\bm{Q}\bm{V}_t\bm{V}_t^{\dagger}\bm{V} = \bm{Q}\bm{V}, 
\]
where we have used that ${\bm{V}_t\bm{V}_t^{\dagger}}$ is a symmetric projector onto $\mathcal{R}(\bm{V}_t) \supset \mathcal{R}(\bm{V}) \cup \mathcal{R}(\bm{QV})$.
}  

\subsection{Hyper-reduction algorithm and target space}
\label{sec:hyperreducalgo}

In this work, we utilize the greedy algorithm for computing empirical cubature points and weights from \cite{an2008optimizing, hernandez2017dimensional}, which is described in Algorithm~\ref{alg:hr}.  
\rnote{Because we assume a fixed reduced basis in time, we compute a single set of empirical cubature points prior to the simulation and use those same points over the entire duration of the simulation.}
Let $\bm{V}_{\rm target}$ be some matrix whose columns span some target space of functions to be integrated, and let $\bm{w}_{\rm target}$ denote some reference target weights.  The algorithm approximates $\bm{V}_{\rm target}^T\bm{w}_{\rm target}$ by $\bm{V}_{\rm target}\LRp{\mathcal{I},:}^T\bm{w}$.  The index set is constructed in a greedy fashion by selecting the row index of $\bm{V}_{\rm target}$ which is most positively parallel to the residual, then computing the weight vector which minimizes the residual through a (non-negative) least squares solve.  The iteration terminates when the norm of the residual is smaller than some user-defined tolerance $tol$.  In this work, we utilize an $N$-mode POD basis, and $tol$ is determined by the singular values of the associated snapshot matrix.  Suppose there are $M$ total singular values of the snapshot matrix; then, we define $tol$ as
\[
tol = \sqrt{\frac{\sum_{j=N+1}^M \sigma_j^2}{\sum_{j=1}^M \sigma_j^2}}.
\]

\begin{algorithm}[!h]
\caption{Compute hyper-reduction s.t.\ $\bm{V}_{\rm target}^T\bm{w}_{\rm target} \approx \bm{V}_{\rm target}\LRp{\mathcal{I},:}^T\bm{w}$ (from \cite{hernandez2017dimensional}).}
\begin{algorithmic}[1]
\STATE \textbf{Input}: $\bm{V}_{\rm target}, \bm{w}_{\rm target}, tol$.
\STATE \textbf{Output}: $\mathcal{I}$, $\bm{w}$.
\STATE Set $\bm{b} = \bm{V}_{\rm target}^T\bm{w}_{\rm target}$, initialize residual $\bm{r} = \bm{b}$, $\mathcal{I} = \emptyset$.
\WHILE {$\nor{\bm{r}}/\nor{\bm{b}} > tol$}
        \STATE Select new index $i$ such that
        \[
        i = \argmin_i \tilde{\bm{V}}_{i}\bm{r}/\nor{\bm{r}}, \qquad
        \tilde{\bm{V}}_i = \bm{V}_{\rm target}(i,:) / \nor{\bm{V}_{\rm target}(i,:)}, \quad i \not\in \mathcal{I}.
        \]
        \STATE Add the new index to the set of hyper-reduced points $\mathcal{I} = \mathcal{I} \cup \LRc{i}$.  
        \STATE Compute $\bm{w}$ using linear least squares
        \[
        \bm{w} = \argmin_{\bm{c}} \frac{1}{2}\nor{\bm{V}_{\rm target}\LRp{\mathcal{I},:}^T\bm{c} - \bm{b}}^2.
        \]
        \IF {$\min \bm{w} \leq 0$}
                \STATE Recompute $\bm{w}$ using non-negative least squares 
                \[
                \argmin_{\bm{c} \geq 0} \frac{1}{2}\nor{\bm{V}_{\rm target}\LRp{\mathcal{I},:}^T\bm{c} - \bm{b}}^2.
                \]             
        \ENDIF
        \STATE Recompute $\bm{r} = \bm{b} - \bm{V}_{\rm target}\LRp{\mathcal{I},:}^T\bm{w}$.
\ENDWHILE
\end{algorithmic}
\label{alg:hr}
\end{algorithm}

Algorithm~\ref{alg:hr} requires both $\bm{V}_{\rm target}$ and $\bm{w}_{\rm target}$ as inputs. For all numerical experiments, we set $\bm{w}_{\rm target} = \Delta x \bm{1}$ and set $\bm{V}_{\rm target}$ to be the span of products of POD functions $\bm{V}$, such that
\begin{equation}
\mathcal{R}\LRp{\bm{V}_{\rm target}} = {\rm span}\LRc{\bm{V}(:,i) \circ \bm{V}(:,j), \quad i,j = 1,\ldots, N}.
\label{eq:Vtarget}
\end{equation}
This choice of target space ensures that the mass matrix $\Delta x \bm{V}^T\bm{V}$ is accurately approximated.  In practice, since many of the pointwise vector products $\bm{V}(:,i) \circ \bm{V}(:,j)$ are linearly dependent, we replace $\bm{V}_{\rm target}$ by a dimensionally reduced matrix \cite{hernandez2017dimensional}.  In this work, we compute this second reduction using another application of the SVD.  Let $\mu_{1}, \ldots, \mu_{N_t}$ denote the singular values of $\bm{V}_{\rm target}$, and define the energy residual of the first $i$ modes
\begin{equation}
E_i = \sqrt{\frac{\sum_{j=i+1}^{N_t} \mu_j^2}{\sum_{j=1}^{N_t} \mu_j^2}}.  
\label{eq:svdenergy}
\end{equation}
Note that $E_1 \geq E_2 \geq \ldots \geq E_{N_t}$.  We replace the matrix $\bm{V}_{\rm target}$ with its $k$ leading left singular vectors, where $k$ is the smallest index such that $E_{k} \leq tol$.  

\subsection{Conditioning of the test mass matrix}
\label{sec:condtest}
We note that the choice of target space is designed to accurately approximate the mass matrix $\Delta x \bm{V}^T\bm{V}$, and does not take the test basis matrix $\bm{V}_{t}$ into account.  This choice is motivated by the empirical observation that incorporating the test basis by approximating either $\Delta x \bm{V}^T\bm{V}_{t}$ or $\Delta x \bm{V}_{t}^T\bm{V}_{t}$ produces a significantly larger number of hyper-reduced points without any significant improvement in accuracy.  
However, because the hyper-reduction strategy does not account for $\bm{V}_{t}$, it is possible to construct a hyper-reduced quadrature which accurately integrates the mass matrix but violates Assumption~\ref{ass:quad} by producing a singular test mass matrix $\bm{V}_t\LRp{\mathcal{I},:}^T\bm{W}\bm{V}_t\LRp{\mathcal{I},:}$.  
The challenge is then to construct a hyper-reduced volume quadrature which accurately approximates the reduced basis mass matrix $\Delta x \bm{V}^T\bm{V}$ while ensuring that the hyper-reduced test mass matrix is non-singular.  We address this challenge by computing the spectra of the test mass matrix  and adding additional ``stabilizing'' points if the condition number is larger than a specified tolerance.  

Let $\bm{z}_j$ denote the $N_z$ eigenvectors corresponding to small eigenvalues of the hyper-reduced test mass matrix.  The entries of $\bm{z}_j$ correspond to vectors of coefficients in the test basis $\bm{V}_t$.  Let $\bm{Z}$ denote the matrix whose columns are point values of each eigenvector
\[
\bm{Z} = \begin{bmatrix}
\bm{V}_t\bm{z}_1 & \ldots & \bm{V}_t\bm{z}_{N_z}
\end{bmatrix}.
\]
To ensure that the hyper-reduced test mass matrix is non-singular, we will employ a greedy heuristic based on adding additional points to approximate $\Delta x\bm{Z}^T\bm{Z}$.  Suppose that we have already computed hyper-reduced points and weights to approximate the reduced mass matrix $\Delta x\bm{V}^T\bm{V}$.   Let $\mathcal{I}_{\bm{V}}$ denote the index set of these stabilizing points.  We compute a second group of hyper-reduced points with index set $\mathcal{I}_{\bm{Z}}$  using the greedy hyper-reduction algorithm applied to the target space $\bm{Z}_{\rm target}$, which we define analogously to $\bm{V}_{\rm target}$ in (\ref{eq:Vtarget}).  We note that, unlike the hyper-reduced points computed previously, we do not perform a dimensionality reduction on $\bm{Z}_{\rm target}$ when computing this set of stabilizing points.

Given $\mathcal{I}_{\bm{Z}}$, we update the hyper-reduced index set as the union $\mathcal{I} = \mathcal{I}_{\bm{V}} \cup \mathcal{I}_{\bm{Z}}$, and compute a new set of hyper-reduced weights through the non-negative least squares solve
\[
\bm{w} = \argmin_{\bm{c} \geq 0} \frac{1}{2} \LRp{\nor{\bm{V}_{\rm target}\LRp{\mathcal{I},:}^T\bm{c} - \bm{b}}^2 + \alpha_Z\nor{\bm{Z}_{\rm target}\LRp{\mathcal{I},:}^T\bm{c} - \bm{d}}^2}, \qquad \bm{d} = \bm{Z}_{\rm target}^T\bm{w}_{\rm target},
\]  
where $\alpha_Z > 0$ is some scaling parameter which controls whether the weights $\bm{w}$ more accurately approximate the integration of the target space $\bm{V}_{\rm target}$ for the mass matrix, or $\bm{Z}_{\rm target}$ for the null space of the test mass matrix.  We set $\alpha_Z = 10^{-2}$ in all experiments, which is empirically observed to control the condition number of the test mass matrix without significantly impacting the accuracy to which the reduced mass matrix $\Delta x\bm{V}^T\bm{V}$ is approximated.  

For the numerical experiments in this paper, it was sufficient to add one single set of stabilizing hyper-reduced points.  However, this procedure can be repeated multiple times to further improve the condition number of the hyper-reduced test mass matrix.  




\subsection{An entropy conservative hyper-reduced model on periodic domains}

To summarize, an entropy conservative reduced order model on periodic domains requires the following offline steps:
\begin{enumerate}
\item Compute a reduced basis matrix $\bm{V}$ from snapshots of both conservative variables and entropy variables.  
\item \bnote{Compute a ``test'' basis matrix $\bm{V}_t$ such that $\mathcal{R}\LRp{\bm{V}_t} = \mathcal{R}\LRp{\LRs{\bm{1}, \bm{V},\bm{Q}\bm{V}}}$}, and use $\bm{V}_t$ to construct a ``modal'' differentiation matrix $\hat{\bm{Q}}_t = \bm{V}_t^T\bm{Q}\bm{V}_t$.  
\item Compute a hyper-reduced quadrature using Algorithm~\ref{alg:hr}, adding stabilizing points as necessary to ensure that the test mass matrix $\bm{V}_t\LRp{\mathcal{I},:}^T\bm{W}\bm{V}_t\LRp{\mathcal{I},:}$ is non-singular.  
\item Construct the hyper-reduced nodal differentiation matrix $\bm{Q}_t = \bm{P}_t^T\hat{\bm{Q}}_t\bm{P}_t$ using the projection matrix $\bm{P}_t = \LRp{\bm{V}_t\LRp{\mathcal{I},:}^T\bm{W}\bm{V}_t\LRp{\mathcal{I},:}}^{-1}\bm{V}_t\LRp{\mathcal{I},:}^T\bm{W}$ onto the test basis.
\end{enumerate}
Then, we have the following theorem:
\begin{theorem}
Suppose that the hyper-reduced quadrature satisfies Assumption~\ref{ass:quad}.  Then, the semi-discrete formulation 
\begin{gather}
\bm{M}_N\td{\bm{u}_N}{t} + 2\bm{V}\LRp{\mathcal{I},:}^T\LRp{\bm{Q}_t \circ \bm{F}}\bm{1} = \bm{0} \label{eq:hrrom}\\
\bm{M}_N = \bm{V}\LRp{\mathcal{I},:}^T\bm{W}\bm{V}\LRp{\mathcal{I},:}, \qquad \bm{P} = \bm{M}_N^{-1}\bm{V}\LRp{\mathcal{I},:}^T\bm{W},\nonumber\\ 
\bm{v}_N = \bm{P}\bm{v}\LRp{\bm{V}\LRp{\mathcal{I},:}\bm{u}_N}, \qquad \tilde{\bm{v}} = \bm{V}\LRp{\mathcal{I},:}\bm{v}_N, \nonumber\\
\tilde{\bm{u}} = \bm{u}\LRp{\tilde{\bm{v}}}, \qquad \bm{F}_{ij} = \bm{f}_S\LRp{\tilde{\bm{u}}_i,\tilde{\bm{u}}_j},\nonumber
\end{gather}
semi-discretely conserves the sampled and weighted average entropy  
\[
\bm{1}^T\bm{W}\td{S\LRp{\bm{V}\LRp{\mathcal{I},:}\bm{u}_N}}{t} = 0.  
\]
Additionally, if the reduced basis $\bm{V}$ exactly represents $\bm{1}$, solutions of (\ref{eq:hrrom}) conserve global averages of the conservative variables.
\label{thm:hrromes}
\end{theorem}
\begin{proof}
Testing the time derivative with $\bm{v}_N$ yields
\begin{gather*}
\bm{v}_N^T\bm{M}_N\td{\bm{u}_N}{t} = \bm{v}\LRp{\bm{V}\LRp{\mathcal{I},:}\bm{u}_h}^T\bm{P}^T\bm{M}_N\td{\bm{u}_N}{t} = \bm{v}\LRp{\bm{V}\LRp{\mathcal{I},:}\bm{u}_h}^T \bm{W}\bm{V}\LRp{\mathcal{I},:}\bm{M}_N^{-1}\bm{M}_N\td{\bm{u}_N}{t}
\\
= \bm{v}\LRp{\bm{V}\LRp{\mathcal{I},:}\bm{u}_h}^T \bm{W} \bm{V}\LRp{\mathcal{I},:}\td{\bm{u}_N}{t} = \bm{v}\LRp{\bm{V}\LRp{\mathcal{I},:}\bm{u}_h}^T \bm{W} \td{\LRp{\bm{V}\LRp{\mathcal{I},:}\bm{u}_N}}{t} = \bm{1}^T\bm{W}\td{S\LRp{\bm{V}\LRp{\mathcal{I},:}\bm{u}_N}}{t},
\end{gather*}
where we have used that $\bm{W}$ is diagonal and the temporal chain rule in the final step.  
Since the hyper-reduction preserves the skew-symmetry and zero row sums of $\bm{Q}_t$, the remainder of the proof consists of showing that $\bm{v}_N^T\bm{V}\LRp{\mathcal{I},:}^T\LRp{\bm{Q}_t \circ \bm{F}}\bm{1} = 0$, and is identical in structure to the proof of Theorem~\ref{thm:esrom}.  
\end{proof}

\begin{remark}
It is not strictly necessary to use the hyper-reduced mass matrix 
\[
\bm{M}_N = \bm{V}\LRp{\mathcal{I},:}^T\bm{W}\bm{V}\LRp{\mathcal{I},:}
\]
 to construct an entropy conservative ROM.  For example, we can use the exact mass matrix $\bm{M}_N = \Delta x \bm{V}^T\bm{V}$ rather than the hyper-reduced mass matrix.  If the projection matrix $\bm{P}$ is defined accordingly, the resulting ROM remains entropy conservative.  However, we observe significantly larger errors when using the exact mass matrix instead of the hyper-reduced mass matrix.  
\end{remark}

\subsection{Online cost estimates for the hyper-reduced system}
\label{sec:cost}
In practice, the flux matrix $\bm{F}$ is not computed explicitly, but is instead formed on the fly by evaluating the entropy stable flux $\bm{f}_S$.  Due to the complexity of the formulas for $\bm{f}_S$, the computation of $\bm{F}$ constitutes a significant portion of the computational cost.  For the full order model, this cost is typically ameliorated by exploiting the sparsity of $\bm{Q}$.  However, since the hyper-reduced matrix $\bm{Q}_t$ is fully dense, the nonlinear convective term $\bm{Q}_{\bnote{t}}\circ\bm{F}$ in (\ref{eq:hrrom}) requires $O(N_s^2)$ nonlinear flux evaluations.  In contrast, typical hyper-reduction approaches target a nonlinear \textit{vector}, which requires only $O(N_s)$ nonlinear function evaluations.  We emphasize that this $O(N_s^2)$ online cost is specific only to the convective nonlinearity, and not to the hyper-reduced treatment of other nonlinear terms such as viscosity.

While the $O(N_s^2)$ nonlinear flux evaluations are trivially parallelizable, they constitute a significant increase in online costs compared to traditional hyper-reduction procedures.  However, we note that the cost can be comparable to techniques for exactly treating higher order polynomial nonlinearities for a large number of modes \cite{maboudi2018conservative}.  For example, suppose one wishes to compute $\bm{V}^T g\LRp{\bm{V}\bm{u}_N}$, where $g(x) = x^m$ for $m \geq 1$ and $m$ integer.  We can rewrite this as
\[
\bm{V}^T g(\bm{u}) = \bm{V}^T \diag{\bm{V}\bm{u}_N}^{m-1} \bm{V}\bm{u}_N.
\]
Suppose first that $m = 2$, such that the nonlinearity is quadratic.  Define $\bm{V}_j = \diag{\bm{V}\LRp{:,j}}$ as the diagonal matrix with the $j$th column of $\bm{V}$ on the diagonal.  Then, $\diag{\bm{V}\bm{u}_N} = \sum_{j=1}^N\bm{V}_j (\bm{u}_N)_j$, and 
\[
\bm{V}^T g\LRp{\bm{V}\bm{u}_N} = \sum_{j=1}^N (\bm{u}_N)_j \bm{V}^T \bm{V}_j \bm{V}\bm{u}_N.
\]
The matrices $\bm{V}^T \bm{V}_j \bm{V}$ can be precomputed, such that the cost of evaluating the above expression is $N$ matrix-vector multiplications, or $O(N^3)$ operations.  For $m = 3$, repeating the same steps yields
\[
\bm{V}^T g\LRp{\bm{V}\bm{u}_N} = \sum_{j=1}^N \sum_{k=1}^M (\bm{u}_N)_j  (\bm{u}_N)_k \bm{V}^T \bm{V}_j \bm{V}_k \bm{V}\bm{u}_N.
\]
which requires $N^2$ matrix-vector multiplications, or $O(N^4)$ operations.  Generalizing this procedure to $m > 3$ yields that a degree $m$ polynomial nonlinearity can be evaluated in $O(N^{m+1})$ operations.\footnote{\bnote{Because entropy conservative fluxes for the compressible Euler equations include non-polynomial rational and logarithmic terms, it is not possible to apply this exact treatment of polynomial nonlinearities to general entropy conservative discretizations. }}   

Experimental results suggest that the number of hyper-reduced points $N_s$ scales as $O(N)$, though the constant increases with the spatial dimension (numerical experiments suggest the constant is approximately 2 in 1D and around 10 in 2D).  Define constants $\alpha = N_s/ N$ and $\beta$, where $\beta$ denotes the number of operations required to evaluate the flux $\bm{f}_S$.  We note that different choices of entropy stable flux functions (see, for example, \cite{ismail2009affordable, chandrashekar2013kinetic}) produce significantly different values of $\beta$.  Then, the hyper-reduced treatment of the convective nonlinearity requires $N_s^2$ flux evaluations.  Assuming flux evaluations dominate costs, entropy stable schemes require $\alpha^2 \beta N^2$ operations.  Thus, rather than increasing the asymptotic computational complexity, the entropy stable hyper-reduction presented here increases the associated constant.  


\subsection{Hyper-reduction and entropy dissipation}
\label{sec:diss}

Given an entropy conservative ROM with hyper-reduction, we can construct an entropy stable ROM by adding appropriate entropy-dissipative viscosity terms.  Recall that the artificial diffusion matrix $\bm{K}$ in (\ref{eq:esvisc}) does not depend nonlinearly on $\bm{u}$.  Thus, typical model reduction techniques compress such operators using Galerkin projection, e.g., $\bm{V}^T\bm{K}\bm{V}$.  However, since one cannot in general show that this operator dicretely dissipates entropy, we discuss two treatments of the viscous terms which are provably entropy dissipative.  

This section introduces hyper-reduction procedures for $\bm{K}$ which provide a provable discrete dissipation of entropy.  We seek to mimic the entropy balance of the full order model in (\ref{eq:entropybalancefom}).  Suppose the reduced order model is given by 
\[
\bm{M}_N\td{\bm{u}_N}{t} + 2\bm{V}\LRp{\mathcal{I},:}^T\LRp{\bm{Q}_t \circ \bm{F}}\bm{1} + \bnote{\epsilon}\bm{d}(\bm{u}_N)= \bm{0}
\]
where $\bm{d}(\bm{u})$ is some hyper-reduced approximation to the viscous terms.  Theorem~\ref{thm:hrromes} gives that the entropy balance of the reduced order model is
\[
\Delta x \bm{1}^T\td{S\LRp{\bm{V}\bm{u}_N}}{t} = -\bnote{\epsilon}\bm{v}_N^T\bm{d}(\bm{u}_N), \qquad \bm{v}_N = \bm{P}\bm{v}\LRp{\bm{V}\LRp{\mathcal{I},:}\bm{u}_N}.
\]
We will design hyper-reduced treatments of the viscous terms such that $\bm{v}_N^T\bm{d}(\bm{u}_N) \geq 0$, which implies that the reduced model is entropy stable.  

Let $\bm{D}$ denote the $(K-1)\times K$ difference matrix
\[
\bm{D} = \frac{1}{\Delta x}\begin{bmatrix}
1 & -1 &&& \\
 & 1 & -1 && \\
 & & \ddots & \ddots&\\
&  & & 1 & -1
\end{bmatrix}.
\]
The artificial diffusion matrix in (\ref{eq:esvisc}) can be decomposed as $\bm{K} = \bm{D}^T\bm{D}$.  To construct hyper-reduced diffusion terms, we first compute (using the SVD) an auxiliary basis $\bm{V}_{\bm{D}}$ for $\bm{D}\bm{V}$, where $\bm{V}$ is the reduced basis matrix.  We then perform hyper-reduction to compute weights and sampled row indices such that 
\[
\bm{V}_{\bm{D}}^T\bm{V}_{\bm{D}} \approx \bm{V}_{\bm{D}}\LRp{\mathcal{I}_{\bm{D}},:}^T \bm{W}_{\bm{D}}\bm{V}_{\bm{D}}\LRp{\mathcal{I}_{\bm{D}},:}.
\]

We can now construct entropy-dissipative diffusion terms using this second hyper-reduction of $\bm{D}$.  The first treatment is relatively straightforward, and simply samples rows of $\bm{D}$ corresponding to indices $\mathcal{I}_{\bm{D}}$.  The viscous terms $\bm{K}\bm{u}$ in the full order model are approximated using the sparse hyper-reduced matrix $\bm{D}\LRp{\mathcal{I}_{\bm{D}},:}$
\begin{equation}
\bm{d}(\bm{u}_N) =\bnote{\bm{V}^T} \bm{D}\LRp{\mathcal{I}_{\bm{D}},:}^T \bm{W}_{\bm{D}} \bm{D}\LRp{\mathcal{I}_{\bm{D}},:}\tilde{\bm{u}}.
\label{eq:visc1}
\end{equation}
Here, the entropy-projected conservative variables $\tilde{\bm{u}}$ are mappings of grid values of the projected entropy variables  $\bm{V}\bm{v}_N$.  Repeating the steps used to derive (\ref{eq:entropydissfom}) yields that the entropy dissipation of (\ref{eq:visc1}) is 
\begin{align*}
\bm{v}_N^T\bm{V}^T\bm{D}^T\LRp{\mathcal{I}_{\bm{D}},:}^T \bm{W}_{\bm{D}} \bm{D}\LRp{\mathcal{I},:}\tilde{\bm{u}} &= 
\tilde{\bm{v}}^T\bm{D}^T\LRp{\mathcal{I}_{\bm{D}},:}^T \bm{W}_{\bm{D}} \bm{D}\LRp{\mathcal{I},:}\bm{u}\LRp{\tilde{\bm{v}}}
\\
&= \sum_{i\in \mathcal{I}_{\bm{D}}} \frac{w_{\bm{D},i}}{\Delta x^2} \LRl{\pd{\bm{u}}{\bm{v}}}_{i,i+1} \LRp{\tilde{\bm{v}}_i - \tilde{\bm{v}}_{i+1}}^2.
\end{align*}

The second treatment of the diffusion terms is motivated by approaches taken in \cite{carpenter2014entropy, chen2017entropy, zakerzadeh2017entropy, upperman2019entropy}.  Recall that, using the mean value theorem, the $i$th entry of $\bm{D}\bm{u}$ is
\begin{equation}
\LRp{\bm{D}\bm{u}}_i = \bm{u}_i - \bm{u}_{i+1} = \LRl{\pd{\bm{u}}{\bm{v}}}_{i,i+1} \LRp{\bm{v}(\bm{u}_i)-\bm{v}(\bm{u}_{i+1})} = \LRl{\pd{\bm{u}}{\bm{v}}}_{i,i+1}  \LRp{\bm{D}\bm{v}(\bm{u})}_i,
\label{eq:meanval}
\end{equation}
where $\LRl{\pd{\bm{u}}{\bm{v}}}_{i,i+1}$ is $\pd{\bm{u}}{\bm{v}}$ evaluated at some intermediate state between $\bm{u}_i$ and $\bm{u}_{i+1}$.  This motivates a hyper-reduction based on sampling and weighting rows of $\bm{D}$ applied to the vector of entropy variables $\bm{v}(\bm{u})$.  The diffusion terms can be approximated by 
\begin{equation}
\bm{d}(\bm{u}_N) = \bm{V}^T\bm{D}\LRp{\mathcal{I}_{\bm{D}},:}^T \bm{W}_{\bm{D}} \bm{H} \bm{D}\LRp{\mathcal{I}_{\bm{D}},:}\bm{V} \bm{v}_N
\label{eq:visc2}
\end{equation}
where $\bm{v}_N$ are coefficients of the projected entropy variables in (\ref{eq:hrrom}), $\bm{W}_{\bm{D}}$ is a diagonal matrix of positive weights $w_{\bm{D},i}$, and $\bm{H}$ is a diagonal matrix containing entries of the Jacobian $\pd{\bm{u}}{\bm{v}}$\footnote{Recall that matrices are treated as Kronecker products, and that the diagonal entries of $\bm{H}_{ii}$ correspond to block matrices acting on the vector of solution components at a point.} evaluated at all hyper-reduced points in $\mathcal{I}_{\bm{D}}$ 
\[
\bm{H}_{ii} = \LRl{\pd{\bm{u}}{\bm{v}}}_{\bm{u} = \bar{\bm{u}}_i}, \qquad \bar{\bm{u}}_i = \frac{\tilde{\bm{u}}_i + \tilde{\bm{u}}_{i+1}}{2}.
\]
Since we do not know the intermediate state used to evaluate $\LRl{\pd{\bm{u}}{\bm{v}}}_{i,i+1}$ in (\ref{eq:meanval}), 
we have arbitrarily chosen to evaluate $\pd{\bm{u}}{\bm{v}}$ at the average of the entropy-projected conservative values $\tilde{\bm{u}}_i$ and $\tilde{\bm{u}}_{i+1}$.  Since the diagonal entries of $\bm{W}_{\bm{D}}$ are positive and $\pd{\bm{u}}{\bm{v}}$ is symmetric positive definite (for physically relevant values of $\tilde{\bm{u}}$), the entropy dissipation of the hyper-reduced diffusion term is given by
\[
\bm{v}_N^T\bm{V}^T\bm{D}^T\LRp{\mathcal{I}_{\bm{D}},:}^T \bm{W}_{\bm{D}} \bm{H} \bm{D}\LRp{\mathcal{I},:}\bm{V} \bm{v}_N = 
\sum_{i\in \mathcal{I}_{\bm{D}}} \frac{w_{\bm{D},i}}{\Delta x^2} \LRl{\pd{\bm{u}}{\bm{v}}}_{\bm{u}=\bar{\bm{u}}_i} \LRp{\tilde{\bm{v}}_i - \tilde{\bm{v}}_{i+1}}^2.
\]
This implies that the dissipation of entropy produced by the hyper-reduced approximation of the viscous terms (\ref{eq:visc2}) also mimics the dissipation of entropy (\ref{eq:entropydissfom}) derived for the full order model.
The difference between the hyper-reduction treatments of viscosity (\ref{eq:visc1}) and (\ref{eq:visc2}) are that (\ref{eq:visc1}) relies explicitly on the fact that $\bm{D}$ is a two-point finite volume difference matrix, which enables the use of the mean value theorem in deriving a statement of entropy dissipation.   The first approach (\ref{eq:visc1}) results in a simpler online phase, though it involves larger matrices compared to the second approach.  The second approach (\ref{eq:visc2}) is not limited to finite volume methods, and is entropy dissipative for any choice of differentiation matrix $\bm{D}$.  Both approaches require computing an additional set of hyper-reduced points during the offline phase.  

\begin{remark}
The hyper-reduction treatments of viscosity described in this section can also be applied to nonlinear interface dissipation models, such as local Lax-Friedrichs penalization.  We do not focus on such approximations for this paper, as the amount of dissipation applied is typically proportional to $\Delta x$ and thus changes based on the dimension of the full order model.
\end{remark}

We note that, in practice, a naive approximation of diffusive terms $\bm{K}\bm{u}$ by
\begin{equation}
\bm{d}(\bm{u}_N) = \bm{V}^T\bm{K}\bm{V}\bm{P}\tilde{\bm{u}},
\label{eq:visc3}
\end{equation}
is often effective.  This approach is simpler to implement than (\ref{eq:visc1}), (\ref{eq:visc2}), and does not require a second set of hyper-reduced points.  We are unable to show that this treatment is provably entropy dissipative.  However, in all experiments, we observe that $\bm{v}_N^T\bm{V}^T\bm{K}\bm{V}\bm{P}\tilde{\bm{u}} \geq 0$, which implies that entropy is discretely dissipated.

\section{Weak imposition of boundary conditions}
\label{sec:6}
Until now, we have assumed periodic domains.  In this section, we discuss how to extend the construction of entropy stable ROMs to the non-periodic case.  Boundary conditions are weakly imposed through a numerical flux and a ``hybridized'' summation by parts operator \cite{chan2017discretely}.  

In the non-periodic case, the matrix $\bm{Q}$ is nearly skew-symmetric, such that it satisfies a summation by parts property
\begin{equation}
\bm{Q} = \frac{1}{2}\begin{bmatrix}
-1 & 1 & & \\
-1 & 0 & \ddots &\\
 &\ddots &\ddots &1\\
  & & -1 &1\\
\end{bmatrix}, \qquad \bm{Q} + \bm{Q}^T = \bm{B}_{\rm SBP} = 
\begin{bmatrix}
-1& && \\
&0 &&\\
& &\ddots &\\
& &&1\\
\end{bmatrix}
\label{eq:fomsbp}
\end{equation}
A non-periodic full order model can be expressed as 
\begin{align}
\Delta x \td{\bm{u}_h}{t} + 2\LRp{{\bm{Q}}\circ \bm{F}}\bm{1} + \bm{B}_{\rm SBP}\LRp{\bm{f}^* - \bm{f}(\bm{u})} = \bm{0},
\label{eq:esbc}
\end{align}
where $\bm{f}^* = \LRs{\bm{f}^*_0, 0,\ldots, 0, \bm{f}^*_K}^T$ is a vector containing values of the nonlinear flux at the boundary points.  
Following \cite{chen2017entropy}, a boundary numerical flux $\bm{f}^*$ is entropy stable if 
\begin{equation}
\psi\LRp{\bm{u}} - \bm{v}(\bm{u})^T\bm{f}^* \leq \bm{0}.
\label{eq:esbflux}
\end{equation}
The flux is referred to as entropy conservative if the inequality in (\ref{eq:esbflux}) is an equality.  If an entropy stable or entropy conservative flux $\bm{f}_S(\bm{u}_L,\bm{u}_R)$ is used to evaluate both the volume and surface flux, then the resulting full order model given by (\ref{eq:esbc}) is also entropy stable or entropy conservative \cite{tadmor1987numerical, parsani2015entropy, chen2017entropy}.  

Extending entropy conservative treatments of boundary conditions to the reduced order model is less straightforward due to fact that reduced matrices do not satisfy the same SBP property (\ref{eq:fomsbp}).  If the full order matrix $\bm{Q}$ satisfies the SBP property, then the hyper-reduced matrix $\bm{Q}_t$ satisfies a \textit{generalized} SBP property \cite{fernandez2014generalized} 
\[
\bm{Q}_t + \bm{Q}_t^T = \bm{E}^T\bm{B}\bm{E}, \qquad \bm{Q}_t = \bm{P}_t^T\bm{V}_t^T\bm{Q}\bm{V}_t\bm{P}_t.
\]
where the matrices $\bm{B}$ and $\bm{E}$ encode scaling by normals and evaluation at boundary points.  In 1D, these matrices are given by
\[
\bm{B} = \begin{bmatrix}-1 & \\ & 1\end{bmatrix}, \qquad \bm{E} = \bm{V}_{bt}\bm{P}_t, \qquad \bm{V}_{bt} = \begin{bmatrix}\bm{V}_t(1,:) \\ \bm{V}_t(N,:)\end{bmatrix}.
\]
where $\bm{V}_{bt}$ interpolates test basis functions to boundary points, and $\bm{V}_t(:,1), \bm{V}_t(:,K)$ denote the first and last columns (corresponding to points on the boundary) of the test basis matrix $\bm{V}_t$.  

For linear problems, it is possible to impose energy stable boundary conditions by combining generalized SBP operators with appropriate simultaneous approximation terms (SATs) \cite{fernandez2018simultaneous}.  However, due to the presence of $\bm{E}$ within the boundary term $\bm{E}^T\bm{B}\bm{E}$, the accurate and entropy stable imposition of nonlinear boundary conditions using generalized SBP-SATs remains an open problem \cite{crean2018entropy, chan2018efficient, chenreview}. We address this issue by constructing a \textit{hybridized} SBP operator \cite{chan2017discretely, chan2019skew} (also referred to as a decoupled SBP operator in \cite{chan2017discretely})
\[
\bm{Q}_{h} = \frac{1}{2}\begin{bmatrix}
\bm{Q}_t - \bm{Q}_t^T & \bm{E}^T\bm{B} \\
-\bm{B}\bm{E} & \bm{B}
\end{bmatrix}.
\]
Note that $\bm{Q}_h$ satisfies a block version of the standard summation by parts property by construction
\[
\bm{Q}_h + \bm{Q}_h^T = \bm{B}_h = \begin{bmatrix}
\bm{0} & \\
& \bm{B} \end{bmatrix},  
\]
where $\bm{B}_h$ is diagonal and does not involve the boundary evaluation matrix $\bm{E}$.  

It is less straightforward to show that $\bm{Q}_h\bm{1} = \bm{0}$, which is also necessary for the proof of entropy stability \cite{chan2019skew}.  Note that the full order matrix $\bm{Q}$ in (\ref{eq:fomsbp}) satisfies $\bm{Q}\bm{1} = \bm{0}$.  Then, since $\bm{1}$ is exactly representable in the test basis $\bm{V}_t$, $\bm{V}_t\bm{P}_t\bm{1} = \bm{1}$, and
\[
\bm{Q}_t = \bm{P}_t^T\bm{V}_t^T\bm{Q}\bm{V}_t\bm{P}_t\bm{1} = \bm{P}_t^T\bm{V}_t^T\bm{Q}\bm{1} = \bm{0}.
\]
This implies that $\bm{Q}_h\bm{1}$ simplifies to
\[
\bm{Q}_h \bm{1} =  \frac{1}{2}\begin{bmatrix}
\bm{Q}_t \bm{1} - \bm{Q}_t^T \bm{1} + \bm{E}^T\bm{B} \bm{1} \\
- \bm{B}\bm{E} \bm{1} + \bm{B}
\end{bmatrix}=\frac{1}{2}\begin{bmatrix}
- \bm{Q}_t^T \bm{1} + \bm{E}^T\bm{B} \bm{1} \\
\bm{0}
\end{bmatrix},
\]
where we have used that $\bm{Q}_t\bm{1} = \bm{0}$ and that (\ref{eq:preproduce}) implies $\bm{E}\bm{1} = \bm{1}$.   Here, we abuse notation by letting $\bm{1}$ denote the vector of ones with appropriate dimension.  Furthermore, since $\bm{Q}_t$ satisfies a generalized SBP property,
\[
- \bm{Q}_t^T \bm{1} + \bm{E}^T\bm{B} \bm{1} = \LRp{-\bm{Q}_t^T + \bm{E}^T\bm{B}\bm{E}} \bm{1} = \bm{Q}_t \bm{1} = \bm{0}.
\]
we conclude that $\bm{Q}_h\bm{1} = \bm{0}$.

By replacing $\bm{Q}_t$ in (\ref{eq:hrrom}) with $\bm{Q}_h$, we can impose boundary conditions weakly through a numerical flux $\bm{f}^*$.  We define $\bm{V}_b$ as the matrix which evaluates the reduced basis at boundary points $\LRc{-1,1}$, and define $\bm{V}_h$ as the matrix which evaluates the reduced basis at both boundary points and hyper-reduced volume points
\[
\bm{V}_{b} = \begin{bmatrix}\bm{V}(1,:) \\ \bm{V}(N,:)\end{bmatrix}, \qquad \bm{V}_h = \begin{bmatrix}
\bm{V}\LRp{\mathcal{I},:}\\
\bm{V}_{b}
\end{bmatrix}.
\]
We then have the following theorem:
\begin{theorem}
If the boundary flux $\bm{f}^*$ is entropy stable in the sense of (\ref{eq:esbflux}), then the hyper-reduced ROM 
\begin{gather}
\bm{M}_N\td{\bm{u}_N}{t} + 2\bm{V}_h^T\LRp{\bm{Q}_h \circ \bm{F}}\bm{1} + \bm{V}_{b}^T\bm{B}\LRp{\bm{f}^* - \bm{f}(\tilde{\bm{u}}_b)} = \bm{0} \label{eq:esbchr}\\
\bm{v}_N = \bm{P}\bm{v}\LRp{\bm{V}\LRp{\mathcal{I},:}\bm{u}_N}, \qquad \tilde{\bm{v}} = 
\bm{V}_h
\bm{v}_N, \qquad \tilde{\bm{v}}_b = \bm{V}_b\bm{v}_N, \nonumber\\
\tilde{\bm{u}} = \bm{u}\LRp{\tilde{\bm{v}}}, \qquad \bm{F}_{ij} = \bm{f}_S\LRp{\tilde{\bm{u}}_i,\tilde{\bm{u}}_j}, \qquad  \tilde{\bm{u}}_b = \bm{u}\LRp{\tilde{\bm{v}}_b}\nonumber
\end{gather}
is also entropy stable such that
\[
\bm{1}^T\bm{W}\td{S\LRp{\bm{V}\bm{u}_N}}{t} \leq 0,
\]
with equality holding for an entropy conservative flux.
\label{thm:esbchr}
\end{theorem}
\begin{proof}
The proof is identical in structure to the proof of Theorem 1 in \cite{chan2017discretely}.  We reproduce it here for completeness.
Testing with $\bm{v}_N$ and using the summation by parts property of $\bm{Q}_h$ then yields 
\begin{align*}
\bm{1}^T\bm{W}\td{S\LRp{\bm{V}\bm{u}_N}}{t}  + \tilde{\bm{v}}^T\LRp{\LRp{\bm{Q}_h-\bm{Q}_h} \circ \bm{F}}\bm{1} + \tilde{\bm{v}}_b^T\bm{B}\bm{f}^*  = \bm{0}.
\end{align*}
Here, we have used that, by the consistency of the entropy conservative flux (\ref{eq:esflux}), the diagonal entries of $\bm{F}$ are $\bm{F}_{ii} = \bm{f}\LRp{\tilde{\bm{u}}_i}$.  Thus, since $\bm{B}_h$ is also diagonal,
\[
\LRp{\bm{B}_h\circ\bm{F}}\bm{1} = \bm{B} \bm{f}(\tilde{\bm{u}}_b).
\]
Expanding out the term $\tilde{\bm{v}}^T\LRp{\LRp{\bm{Q}_h-\bm{Q}_h^T} \circ \bm{F}}\bm{1}$ and using skew-symmetry of $\bm{Q}_h-\bm{Q}_h^T$ yields
\begin{align*}
\tilde{\bm{v}}^T\big({\LRp{\bm{Q}_h-\bm{Q}_h^T}} \circ \bm{F}\big)\bm{1} &= 
\frac{1}{2}\sum_{ij} \LRp{\bm{Q}_h-\bm{Q}_h^T}_{ij} \LRp{\tilde{\bm{v}}_i-\tilde{\bm{v}}_j}^T\bm{f}_S\LRp{\tilde{\bm{u}}_i,\tilde{\bm{u}}_j}\\
&= \frac{1}{2}\sum_{ij} \LRp{\bm{Q}_h-\bm{Q}_h^T}_{ij} \LRp{\psi(\tilde{\bm{u}}_i)-\psi(\tilde{\bm{u}}_j)}\\
&= {\psi(\tilde{\bm{u}})^T\bm{Q}_h\bm{1} - \bm{1}^T\bm{Q}_h\psi(\tilde{\bm{u}})} = -\bm{1}^T\bm{B}_h\psi(\tilde{\bm{u}}) = -\bm{1}^T\bm{B}\psi(\tilde{\bm{u}}_b),
\end{align*}
where we have used that $\bm{Q}_h\bm{1} = \bm{0}$ and the SBP property in the last two lines.  Using that $\bm{B}$ is diagonal, straightforward manipulations imply that
\[
\bm{1}^T\bm{W}\td{S\LRp{\bm{V}\bm{u}_N}}{t}  - \bm{1}^T\bm{B}\LRp{\psi(\tilde{\bm{u}}_b)-\tilde{\bm{v}}_b^T\bm{f}^*}  = \bm{0}.
\]
If the flux is entropy stable, then $\bm{1}^T\bm{W}\td{S\LRp{\bm{V}\bm{u}_N}}{t} \leq 0$.  Equality holds if the flux is entropy conservative such that $\psi(\tilde{\bm{u}}_b)=\tilde{\bm{v}}_b^T\bm{f}^*$.
\end{proof}


\section{Extension to higher dimensions}
\label{sec:7}

\subsection{Periodic domains}
For periodic domains, the extension to higher spatial dimensions is relatively straightforward.  We provide a concrete construction of matrices and formulations in two dimensions.  For finite volume methods on a $K \times K$ structured quadrilateral elements of size $\Delta x \times \Delta x$, differentiation matrices along each coordinate direction can be constructed as Kronecker products of one-dimensional matrices.  Let $\bm{Q}_{\rm 1D}$ denote the one-dimensional periodic differentiation matrix defined in (\ref{eq:Qmat}).  The differentiation matrix along the $i$th coordinate direction $\bm{Q}^i$ can be constructed as follows
\[
\bm{Q}^1 = \bm{Q}_{\rm 1D}\otimes \Delta x \bm{I}, \qquad \bm{Q}^2 = \Delta x\bm{I} \otimes \bm{Q}_{\rm 1D}.
\]
Given a basis matrix $\bm{V}$ and indices of hyper-reduced points $\mathcal{I}$, we can introduce test basis matrices $\bm{V}^i_t$ such that the range of  $\bm{V}^i_t$ is equal to the direct sum of the ranges of $\bm{V}$ and $\bm{Q}^i\bm{V}$
\[
\mathcal{R}\LRp{\bm{V}^i_t} = \mathcal{R}\LRp{\begin{bmatrix}\bm{V} &\bm{Q}^i\bm{V}\end{bmatrix}}, \qquad i = 1,\ldots,d.
\]
Note that, unlike the 1D case, the test bases differ along each coordinate direction.  However, the dimension of the range of each test basis $\bm{V}_t^i$ is still at most $2N+1$.  
We can now define hyper-reduced projection matrices for the $i$th coordinate direction
\begin{align*}
\bm{P}^i_t = \LRp{\bm{V}^i_t\LRp{\mathcal{I},:}^T\bm{W}\bm{V}^i_t\LRp{\mathcal{I},:}}^{-1}\bm{V}^i_t\LRp{\mathcal{I},:}^T\bm{W}.
\end{align*}
These matrices can be used to construct reduced nodal differentiation matrices $\bm{Q}_t^i$ 
\[
\bm{Q}_t^i = \LRp{\bm{P}^i_t}^T\LRp{\LRp{\bm{V}^i_t}^T\bm{Q}^i\bm{V}^i_t}\bm{P}^i_t,
\]
which are used to construct an entropy conservative reduced order model in $d$ dimensions
\begin{gather}
\bm{M}_N\td{\bm{u}_N}{t} + \sum_{i=1}^d 2\bm{V}\LRp{\mathcal{I},:}^T\LRp{\bm{Q}^i_t \circ \bm{F}^i}\bm{1} = \bm{0} \label{eq:hrromd}\\
\bm{M}_N = \bm{V}\LRp{\mathcal{I},:}^T\bm{W}\bm{V}\LRp{\mathcal{I},:}, \qquad \bm{F}^i_{jk} = \bm{f}^i_S\LRp{\tilde{\bm{u}}_j,\tilde{\bm{u}}_k}.\nonumber
\end{gather}
Here, $\bm{f}^i_S\LRp{\bm{u}_L,\bm{u}_R}$ is the entropy conservative flux in the $i$th coordinate direction and $\bm{M}_N$ is the mass matrix as defined in (\ref{eq:hrrom}).  We can show that (\ref{eq:hrromd}) is entropy conservative by repeating steps in the proof of Theorem~\ref{thm:hrromes} for each coordinate direction.  

The diffusive matrices $\bm{K}$ can be similarly extended to higher dimensions via Kronecker products
\[
\bm{K} = \bm{K}_{\rm 1D}\otimes \bm{I} + \bm{I}\otimes \bm{K}_{\rm 1D}.  
\]
An hyper-reduction of $\bm{K}$ analogous to the one described in Section~\ref{sec:diss} can be used to guarantee a discrete dissipation of entropy.  

\subsection{Non-periodic domains in higher dimensions}

Non-periodic domains and boundary conditions can be treated by combining hybridized SBP operators and a weak imposition of boundary conditions as was done for the 1D hyper-reduced formulation (\ref{eq:esbchr}).  Without loss of generality, we assume a square domain with $4$ boundary faces, each of which is discretized by $K_{\rm 1D}$ intervals of size $\Delta x$.  Denote values of the outward normal vector on the domain boundary by $\bm{n} = \LRs{n_1, \ldots, n_d}^T$.  We enforce non-periodic boundary conditions through numerical fluxes at boundary points.  However, the number of boundary points scales with $K_{\rm 1D}$.  To ensure that the cost of the reduced order model scales with the number of modes $N$ rather than the dimension of the full order model, we approximate boundary terms by a hyper-reduced weighted combination of sampled boundary points.  

Let $\bm{V}_b$ denote the matrix whose columns correspond to values of the reduced basis at boundary points.  Let $\mathcal{I}_b$ denote a sub-sampled set of $N_b$ boundary points, and let $g$ denote a nonlinear function.  We then approximate the boundary inner product 
\[
\bm{V}_b^Tg(\bm{V}_b\bm{u}_N) \approx \bm{V}_b\LRp{\mathcal{I}_b,:}^T \bm{W}_b g\LRp{\bm{V}_b\LRp{\mathcal{I}_b,:}}.
\]
where $\bm{W}_b = {\rm diag}\LRp{\bm{w}_b}$ is a diagonal matrix whose entries consist of positive weights.  
We also define $N_b\times N_b$ operators $\bm{B}^1, \bm{B}^2$ as the diagonal matrices whose entries consist of weighted component values of outward normals on the boundary 
\[
\bm{B}^i = {\rm diag}\LRp{\bm{n}^i} \bm{W}_b,
\]
where $\bm{n}^i$ is a vector containing the values of the component $n_i$ at boundary points.  

We can now construct hybridized SBP operators using hyper-reduced points on both the volume and boundary of the domain.  Let $\bm{V}_{bt}$ be the matrix which evaluates the $i$th test basis at all points on the boundary.  Then, we can define $\bm{E}_i$ as the matrix which extrapolates from hyper-reduced volume points to hyper-reduced boundary points through projection onto the $i$th test basis
\[
\bm{E}_i = \bm{V}_{bt}^i\LRp{\mathcal{I}_b,:}\bm{P}^i_t.  
\]
The hybridized SBP operator for differentiation along the $i$th coordinate is then
\[
\bm{Q}_h^i = \begin{bmatrix}
\bm{Q}^i_t - \LRp{\bm{Q}^i_t}^T & \bm{E}_i^T\bm{B}^i \\ 
-\bm{B}^i\bm{E}_i & \bm{B}^i
\end{bmatrix}.
\]
These operators can be used to construct a hyper-reduced formulation in higher dimensions
\begin{gather}
\bm{M}_N\td{\bm{u}_N}{t} + \sum_{i=1}^d \LRp{2\bm{V}\LRp{\mathcal{I},:}^T\LRp{\bm{Q}^i_h \circ \bm{F}^i}\bm{1} + \bm{V}_b\LRp{\mathcal{I}_b,:}^T\bm{B}^i\LRp{\bm{f}_i^* - \bm{f}_i(\bm{u}_b)}} = \bm{0} \label{eq:hrromdbc},
\end{gather}
where $\bm{M}_N$ and $\bm{F}^i$ are as defined in (\ref{eq:hrromd}), and $\bm{f}_i(\bm{u}), \bm{f}_i^*$ are the $i$th components of the flux function and boundary numerical flux, respectively.  

In contrast to the 1D case, additional steps are necessary to ensure that entropy stability is preserved under this boundary hyper-reduction.  Recall that the proof of entropy stability for 1D non-periodic domains in Theorem~\ref{thm:esbchr} requires that $\bm{Q}_h\bm{1} = \bm{0}$.  The analogous proof of entropy stability in multiple dimensions requires that $\bm{Q}^i_h\bm{1} = \bm{0}$ for $i = 1,\ldots, d$.  However, this is not automatically satisfied under an arbitrary hyper-reduction of the boundary points.  Expanding out $\bm{Q}^i_h\bm{1}$ yields
\[
\bm{Q}^i_h\bm{1} = \begin{bmatrix}
\bm{Q}^i_t\bm{1} - \LRp{\bm{Q}^i_t}^T\bm{1} + \bm{E}_i^T\bm{B}^i \bm{1}\\
\bm{0}
\end{bmatrix}= \begin{bmatrix}
-\LRp{\bm{Q}^i_t}^T\bm{1} + \bm{E}_i^T\bm{B}^i \bm{1}\\
\bm{0}
\end{bmatrix}.
\]
In general, $-\LRp{\bm{Q}^i_t}^T\bm{1} + \bm{E}^T\bm{B}^i \bm{1}\neq 0$, so $\bm{Q}^i_h\bm{1} \neq \bm{0}$ and the proof of entropy stability does not hold.  To enforce that $\bm{Q}^i_h\bm{1} = \bm{0}$, we impose constraints on the boundary weights $\bm{w}_b$.  Note that 
\begin{align*}
-\LRp{\bm{Q}^i_t}^T\bm{1} + \bm{E}^T\bm{B}^i \bm{1} &= -\LRp{\bm{P}^i_t}^T\bm{V}_t^T\bm{Q}\bm{V}_t\bm{P}^i_t\bm{1} + \LRp{\bm{P}^i_t}^T\LRp{\bm{V}^i_{bt}\LRp{\mathcal{I}_b,:}}^T\bm{B}^i\bm{1}\\
&= \LRp{\bm{P}^i_t}^T\LRp{-\LRp{\bm{V}^i_t}^T\bm{Q}\bm{1} + \LRp{\bm{V}^i_{bt}}\LRp{\mathcal{I}_b,:}^T{\rm diag}\LRp{\bm{n}^i}\bm{w}_b}.
\end{align*}
where we have used that $\bm{V}^i_t\bm{P}^i_t\bm{1} = \bm{1}$.  Thus, to ensure that $\bm{Q}^i_h\bm{1} = 0$, it is sufficient to guarantee that
\begin{equation}
\LRp{\bm{V}^i_{bt}}\LRp{\mathcal{I}_b,:}^T{\rm diag}\LRp{\bm{n}^i}\bm{w}_b = \LRp{\bm{V}^i_t}^T\bm{Q}\bm{1}, \qquad i = 1,\ldots, d.
\label{eq:sbpconstraints}
\end{equation}
The $dN$ constraints encoded in (\ref{eq:sbpconstraints}) can be interpreted as enforcing a discrete version of the fundamental theorem of calculus relating approximate integrals of reduced basis derivatives to boundary integrals of reduced basis values.  We enforce these constraints on the boundary weights $\bm{w}_b$ (\ref{eq:sbpconstraints}) directly into a hyper-reduction approach based on the solution of a linear programming problem using the dual simplex method, and refer the reader to \cite{yano2019lp} for details.  We note that it may also be possible to enforce these constraints, for example, by augmenting the non-negative least squares solve in Algorithm~\ref{alg:hr} with equality constraints, and will explore these directions in future work.

\section{Numerical experiments}
\label{sec:8}
{In this section, we study the behavior of entropy stable reduced order models for} the compressible Euler equations.  In $d$ dimensions, these are given by:
\begin{align*}
\pd{\rho}{t} &+ {\sum_{j=1}^d \pd{\LRp{\rho {u}_j}}{x_j}} = 0,\\
{\pd{\rho {u}_i}{t}} &+ {\sum_{j=1}^d \pd{\LRp{\rho u_iu_j + p\delta_{ij} }}{x_j}} = 0, \qquad i = 1,\ldots,d\\
\pd{E}{t} &+ {\sum_{j=1}^d \pd{\LRp{{u}_j(E+p)}}{x_j}} = 0.\nonumber
\end{align*}
Here, $\rho$ is density, {$u_i$ is the $i$th component of} velocity, and $E$ is the total energy.  The pressure $p$ and specific internal energy $\rho e$ are given by 
\[
{p = (\gamma-1)\LRp{E - \frac{1}{2}\rho \LRb{\bm{u}}^2}}, \qquad {\rho e = E - \frac{1}{2}\rho\LRb{\bm{u}}^2}, \qquad \LRb{\bm{u}}^2 = {\sum_{j=1}^d u_j^2}.
\]
There is a unique entropy which symmetrizes the viscous heat conduction term in the compressible Navier-Stokes equations \cite{hughes1986new}.  This entropy $S(\bm{u})$ is given by
\begin{equation*}
S(\bm{u}) = -\bnote{\rho s}, 
\end{equation*}
where $s = \log\LRp{\frac{p}{\rho^\gamma}}$ is the physical specific entropy, and the dimension $d = 1,2,3$.  The entropy variables in $d$ dimensions are given by
\begin{gather*}
v_1 = \frac{\rho e (\gamma + 1 - s) - E}{\rho e}, \qquad v_{d+2} = -\frac{\rho}{\rho e}, \qquad
v_{1+ i}= \frac{\rho {{u}_i}}{\rho e}, \qquad i = 1,\ldots, d.
\end{gather*}
while the conservation variables in terms of the entropy variables are given by
\begin{gather*}
\rho = -(\rho e) v_{d+2}, \qquad E = (\rho e)\LRp{1 - \frac{\sum_{j=1}^d{v_{1+j}^2}}{2 v_{d+2}}}, \qquad
 \rho {u_i} = (\rho e) v_{1+i}, \qquad i = 1,\ldots,d,
\end{gather*}
where the quantities $\rho e$ and $s$ in terms of the entropy variables are 
\begin{equation*}
\rho e = \LRp{\frac{(\gamma-1)}{\LRp{-v_{d+2}}^{\gamma}}}^{1/(\gamma-1)}e^{\frac{-s}{\gamma-1}}, \qquad s = \gamma - v_1 + \frac{\sum_{j=1}^d{v_{1+j}^2}}{2v_{d+2}}.
\end{equation*}
{Let $f_i, f_j$ denote two arbitrary values.  We define the average and logarithmic average}
\[
{\avg{f} = \frac{f_i + f_j}{2}}, \qquad {\avg{f}^{\log} = \frac{f_i - f_j}{\log\LRp{f_i}-\log\LRp{f_j}}}.
\]
To ensure numerical stability when the denominator is close to zero, we evaluate the logarithmic average using the algorithm of \cite{ismail2009affordable}.  
Explicit expressions for entropy conservative fluxes are given by Chandrashekar \cite{chandrashekar2013kinetic}.  In 1D, these fluxes are
\begin{align*}
f^1_S(\bm{u}_L,\bm{u}_R) &= \avg{\rho}^{\log} \avg{u}\\
f^2_S(\bm{u}_L,\bm{u}_R) &= f^1_S(\bm{u}_L,\bm{u}_R)\avg{u} + p_{\rm avg}\\
f^3_S(\bm{u}_L,\bm{u}_R) &= \LRp{E_{\rm avg} + p_{\rm avg}}\avg{u},
\end{align*}
where the auxiliary quantities are
\begin{gather*}
{\beta = \frac{\rho}{2p}}, \qquad p_{\rm avg} = \frac{\avg{\rho}}{2\avg{\beta}}, \qquad E_{\rm avg} = \frac{\avg{\rho}^{\log}}{2\avg{\beta}^{\log}\LRp{\gamma -1}}   + \avg{\rho}^{\log}\frac{{u_{\rm avg}^2}}{2}.
\end{gather*}
In 1D, $u_{\rm avg}^2 = 2\avg{u}^2-\avg{u^2}$.  
In two dimensions, the $x$ and $y$ fluxes are given by
\begin{align*}
&f^1_{x,S}(\bm{u}_L,\bm{u}_R) = \avg{\rho}^{\log} \avg{u_1},& &f^1_{y,S}(\bm{u}_L,\bm{u}_R) = \avg{\rho}^{\log} \avg{u_2},&\\
&f^2_{x,S}(\bm{u}_L,\bm{u}_R) = f^1_{1,S} \avg{u_1} + p_{\rm avg},&  &f^2_{y,S}(\bm{u}_L,\bm{u}_R) = f^1_{2,S} \avg{u_1},&\nonumber\\
&f^3_{x,S}(\bm{u}_L,\bm{u}_R) = f^2_{2,S},& &f^3_{y,S}(\bm{u}_L,\bm{u}_R) = f^1_{y,S} \avg{u_2} + p_{\rm avg},&\nonumber\\
&f^4_{x,S}(\bm{u}_L,\bm{u}_R) = \LRp{E_{\rm avg} + p_{\rm avg}}\avg{u_1},& &f^4_{y,S}(\bm{u}_L,\bm{u}_R) = \LRp{E_{\rm avg} + p_{\rm avg} }\avg{u_2},& \nonumber
\end{align*}
where in 2D, we redefine ${u^2_{\rm avg}} = 2(\avg{u_1}^2 + \avg{u_2}^2) - \LRp{\avg{u_1^2} +\avg{u_2^2}}.$

The viscous terms are treated as described in Section~\ref{sec:entropydissipation}, and utilize evaluations of the Jacobian $\pd{\bm{u}}{\bm{v}}$.  Explicit expressions for Jacobians in the compressible Euler and Navier-Stokes equations are given in \cite{hughes1986new}.  The solution is evolved using an explicit 5-stage 4th order Runge-Kutta scheme.  

\subsection{1D Euler equations}

We begin by examining the 1D Euler equations with reflective wall boundary conditions.  We utilize a full order finite volume model with $2500$ cells on $[-1,1]$ with a CFL of $.75$, and set the artificial viscosity parameter to $\epsilon = 2\times 10^{-4}$.   It was shown in \cite{svard2014entropy, chen2017entropy} that wall boundary conditions can be imposed using a ``mirror state''.  We use a boundary numerical flux $\bm{f}^* = \bm{f}_S\LRp{\bm{u}^+,\bm{u}}$ augmented with local Lax-Friedrichs penalization, where the exterior state $\bm{u}^+$ is defined by
\[
\rho^+ = \rho, \qquad {u}^+ = -u, \qquad p^+ = p.  
\]  The initial condition is set to be 
\[
\rho = 2 + \frac{1}{2}e^{-100\LRp{x-\frac{1}{2}}^2}, \qquad 
    u = \frac{1}{10}e^{-100\LRp{x-\frac{1}{2}}^2}, \qquad p = \rho^{\gamma}.
\]
The solution exhibits a viscous shock some time after $T=.25$.  

We first examine the impact of enriching solution snapshots with snapshots of the entropy variables.  We run the full order model until final time $T=.7$ and store the solution at $2801$ time-steps.  We compute a reduced basis using SVD by subsampling every $10$ snapshots.  Figure~\ref{subfig:a} shows the decay of the singular values with and without entropy variable enrichment.  In both cases, the decay is similar, with entropy variable enrichment resulting in a slightly slower decay for singular values smaller than $10^{-8}$.  The difference becomes more pronounced for a coarser subsampling of the snapshots.  We note that the enrichment does affect the form of the resulting singular vectors (Figure~\ref{subfig:c}).  We also compute projection errors for each solution snapshot using $25$ modes in Figure~\ref{subfig:b}, along with differences between the projected solutions with and without entropy variable enrichment.  The projection errors are nearly indistinguishable.  

\begin{figure}
\centering
\subfloat[Singular values]{\raisebox{.24em}{\includegraphics[width=.33\textwidth]{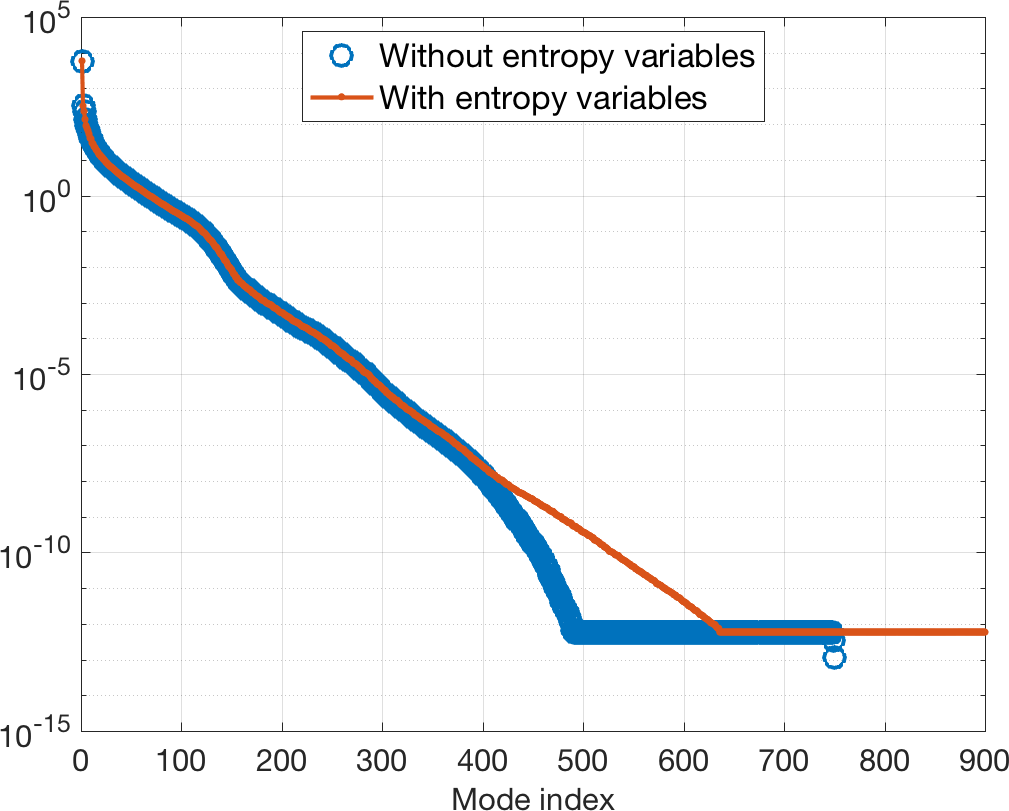}}\label{subfig:a}}
\hspace{.02em}
\subfloat[Fifth singular vector]{\raisebox{.6em}{\includegraphics[width=.322\textwidth]{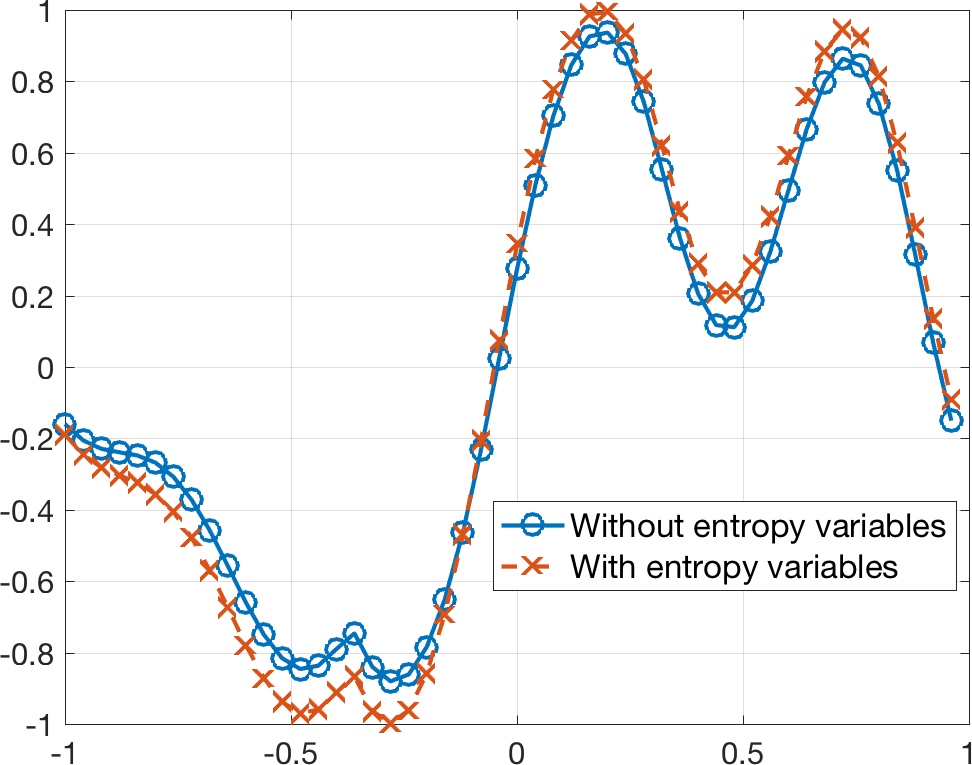}}\label{subfig:c}}
\hspace{.02em}
\subfloat[Projection errors, 25 modes]{\raisebox{.05em}{\includegraphics[width=.33\textwidth]{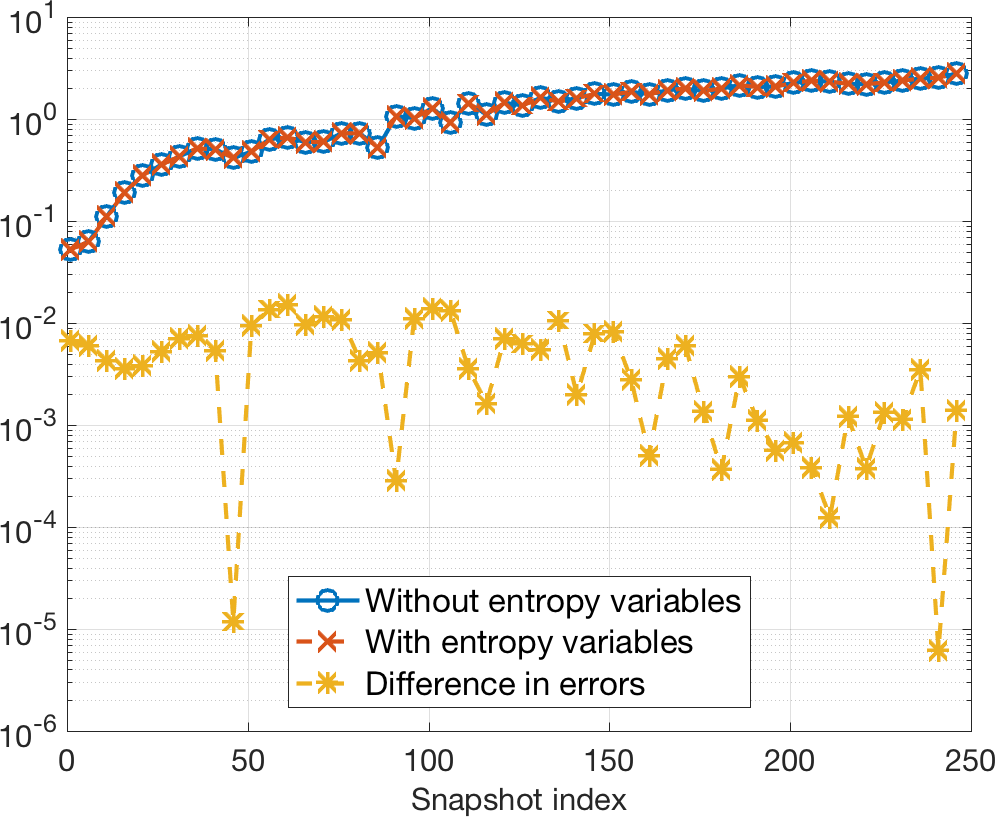}\label{subfig:b}}}
\caption{Snapshot singular values and reduced basis functions with and without entropy variable enrichment. }
\label{fig:svd}
\end{figure}

We next examine solutions produced by the reduced order models.  It is known that reduced models can sometimes utilize a larger CFL compared the full order model \cite{lucia2004reduced, knezevic2011reduced}; however, to ensure that temporal errors are small, all solutions are computed using the same CFL used to generate the solution snapshots.  Figure~\ref{fig:romsols} shows density computed using 25, 75, and 125 modes and the hyper-reduced treatment of viscosity in (\ref{eq:visc2}) \bnote{using $\epsilon = 2e-4$}.  For all resolutions, the shock is under-resolved and the solution possesses Gibbs-type oscillations.  However, despite this under-resolution, the ROM remains stable and does not blow up.  Moreover, the reduced order solution converges uniformly to the full order solution as the number of modes increases.  

\begin{figure}[!h]
\centering
\subfloat[25 modes, $T=.25$]{\includegraphics[width=.32\textwidth]{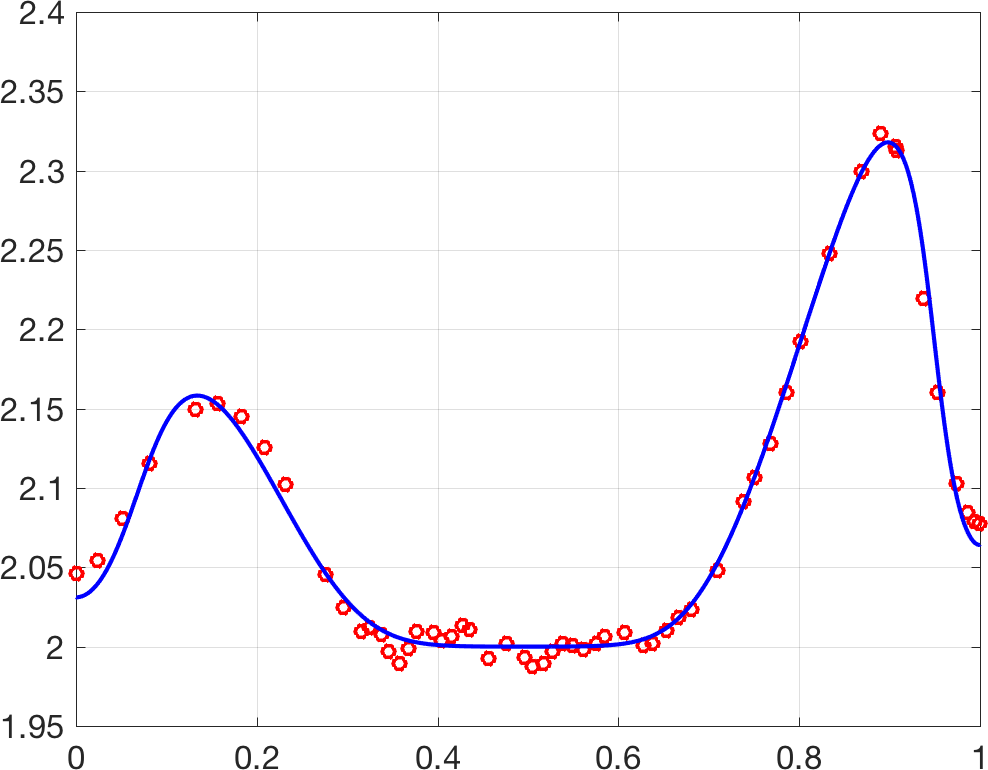}}
\hspace{.05em}
\subfloat[75 modes, $T=.25$]{\includegraphics[width=.32\textwidth]{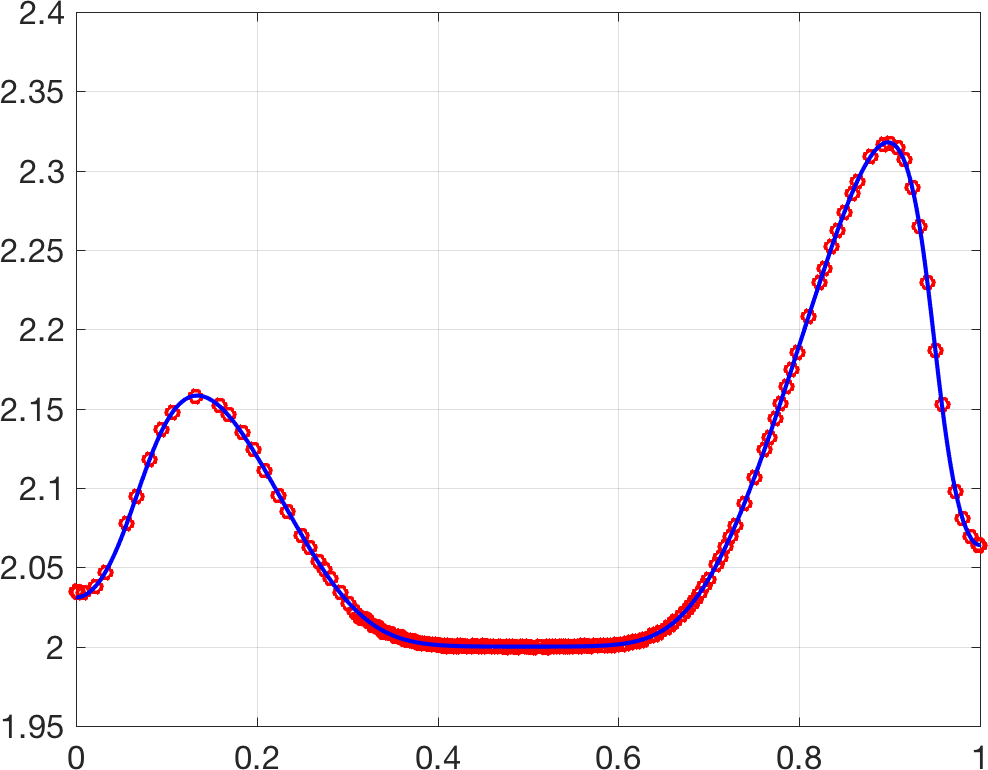}}
\hspace{.05em}
\subfloat[125 modes, $T=.25$]{\includegraphics[width=.32\textwidth]{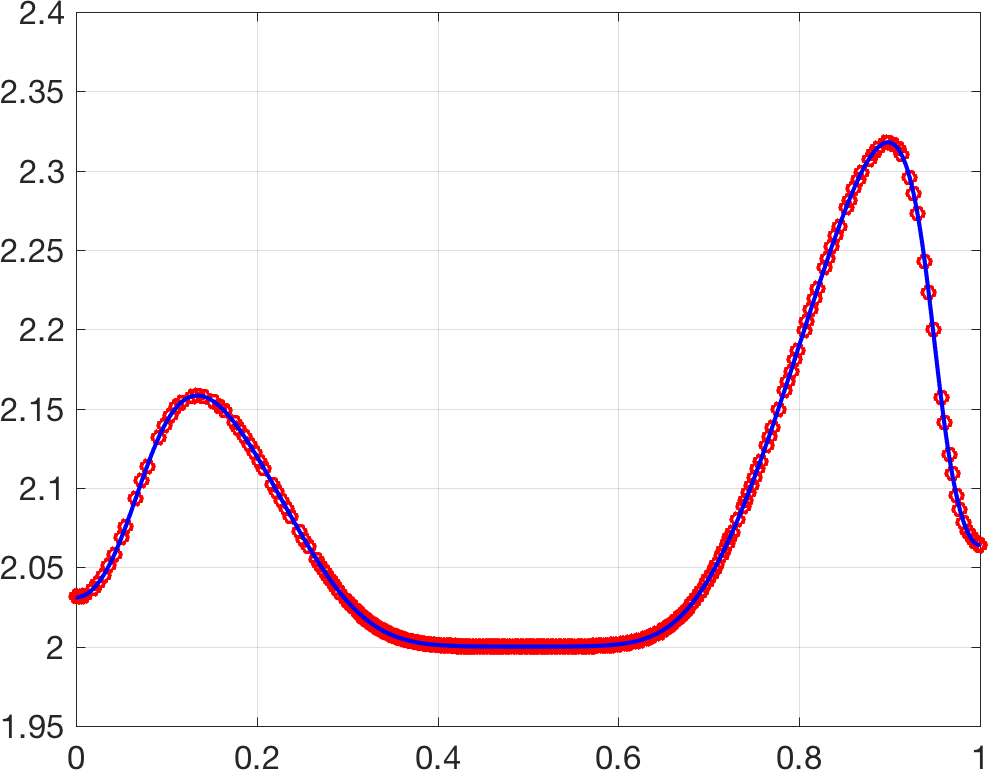}}\\
\subfloat[25 modes, $T=.75$]{\includegraphics[width=.32\textwidth]{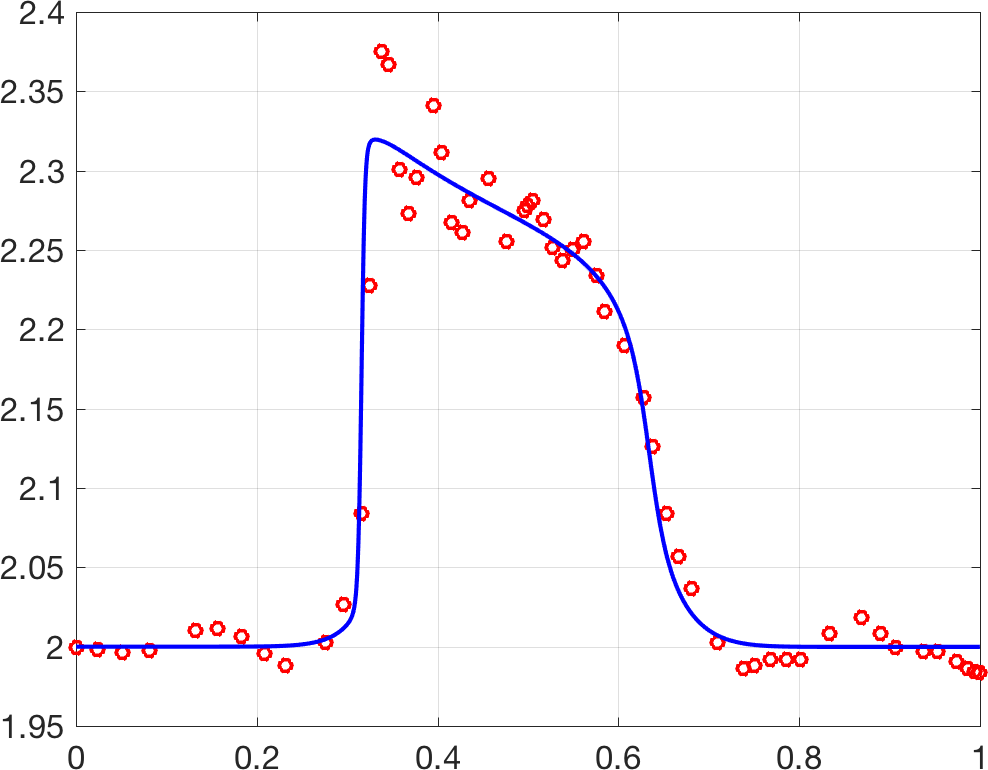}}
\hspace{.05em}
\subfloat[75 modes, $T=.75$]{\includegraphics[width=.32\textwidth]{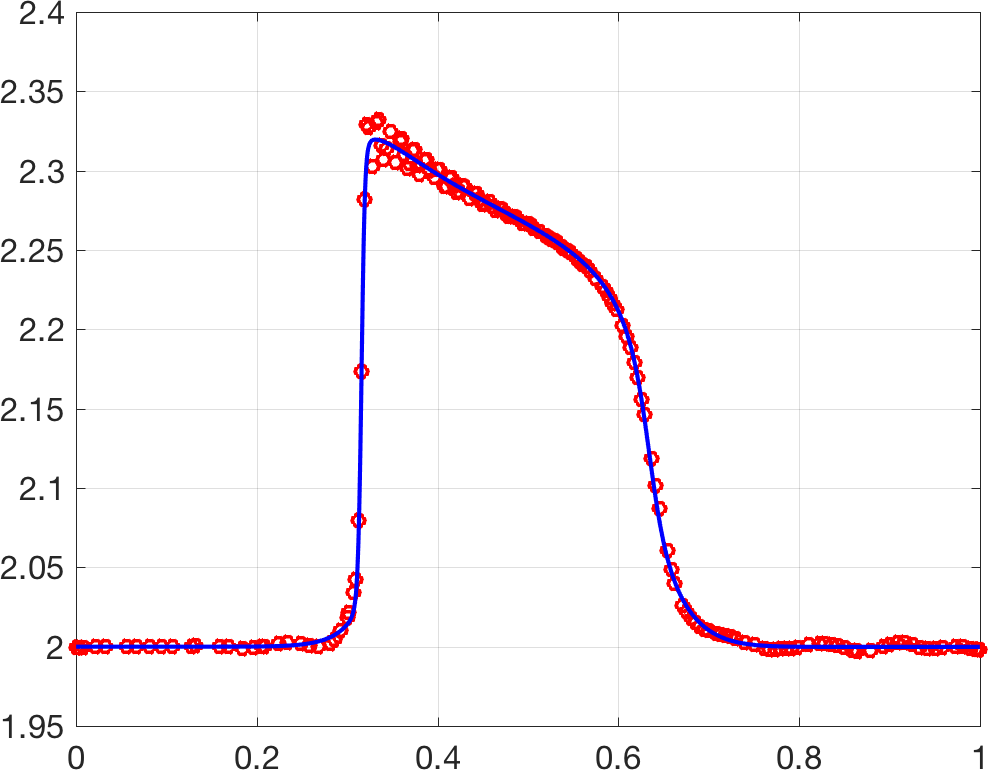}}
\hspace{.05em}
\subfloat[125 modes, $T=.75$]{\includegraphics[width=.32\textwidth]{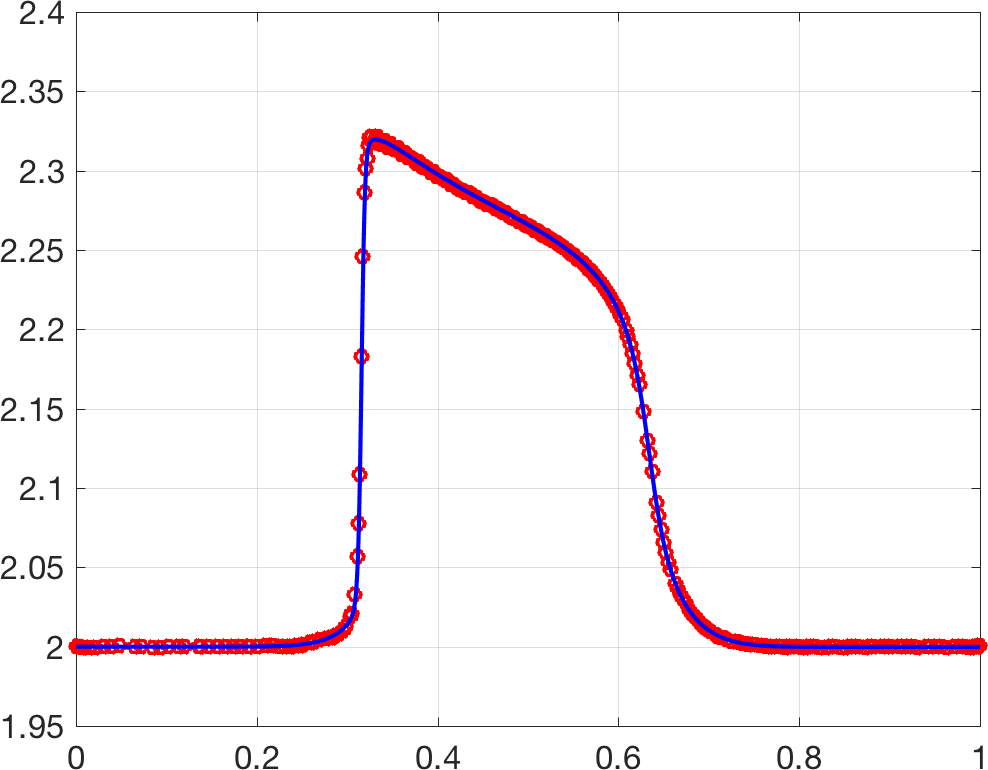}}
\caption{Density $\rho$ computed using 25, 75, and 125 modes at times $T=.25, .75$ \bnote{for $\epsilon = 2e-4$}.  The solution of the reduced order model is plotted at convective hyper-reduced points, shown in \textcolor{red}{red}.  The viscous points are not shown.}
\label{fig:romsols}
\end{figure}

Table~\ref{tab:hrpts} shows the number of hyper-reduced points computed for different numbers of modes.  We report on three different sets of points: empirical cubature points (used to approximate the mass matrix $\Delta x \bm{V}^T\bm{V}$), stabilizing points (added to control the condition number of the test mass matrix; see Section~\ref{sec:condtest}), and viscous points (used to approximate viscous terms; see Section~\ref{sec:diss}).  All point sets are computed using Algorithm~\ref{alg:hr} with dimensionality reduction as described in Section~\ref{sec:hyperreducalgo}.  In all cases, we observe that the number of empirical cubature and viscous points grows roughly as $2N$ (similar observations were made in \cite{hernandez2017dimensional}), while the number of stabilizing points hovers around $30$ for $N > 25$.  

\begin{table}[!h]
\centering
\begin{tabular}{|c || c | c | c |c |}
\hline
Number of modes $N$ & 25 & 75 & 125 & 175\\
\hhline{|=|=|=|=|=|}
Number of empirical cubature points &54 & 158 & 259 &355 \\
\hline
Number of stabilizing points & 3& 21 & 36 & 28 \\
\hline
Number of viscous points & 54 &159 & 259 & 366  \\
\hline
\end{tabular}
\caption{Number of computed hyper-reduced points for the 1D Euler equations.}
\label{tab:hrpts}
\end{table}

Next, we verify that (in the absence of viscosity), the ROM is semi-discretely entropy conservative.  The proof of this property uses that the entropy contribution of the convective term $\bm{v}_N^T \bm{V}_h^T\LRp{\bm{Q}_h \circ \bm{F}}\bm{1} = 0$.  Figure~\ref{fig:novisc} verifies this, plotting the absolute value of the entropy convective term over time.  At all time-steps, $\LRb{\bm{v}_N^T \bm{V}_h^T\LRp{\bm{Q}_h \circ \bm{F}}\bm{1}}$ is near $O\LRp{10^{-14}}$ (close to machine precision), confirming that the formulation (\ref{eq:esbchr}) is discretely entropy conservative in the absence of viscosity and boundary flux penalization terms.  For 125 modes, the solution is significantly more oscillatory than the case of $\epsilon = 2\times 10^{-4}$; however, the reduced solution does not blow up.  

\begin{figure}[!h]
\centering
\subfloat[Convective entropy contribution $\LRb{\bm{v}_N^T \bm{V}_h^T\LRp{\bm{Q}_h \circ \bm{F}}\bm{1}}$]{\includegraphics[width=.4525\textwidth]{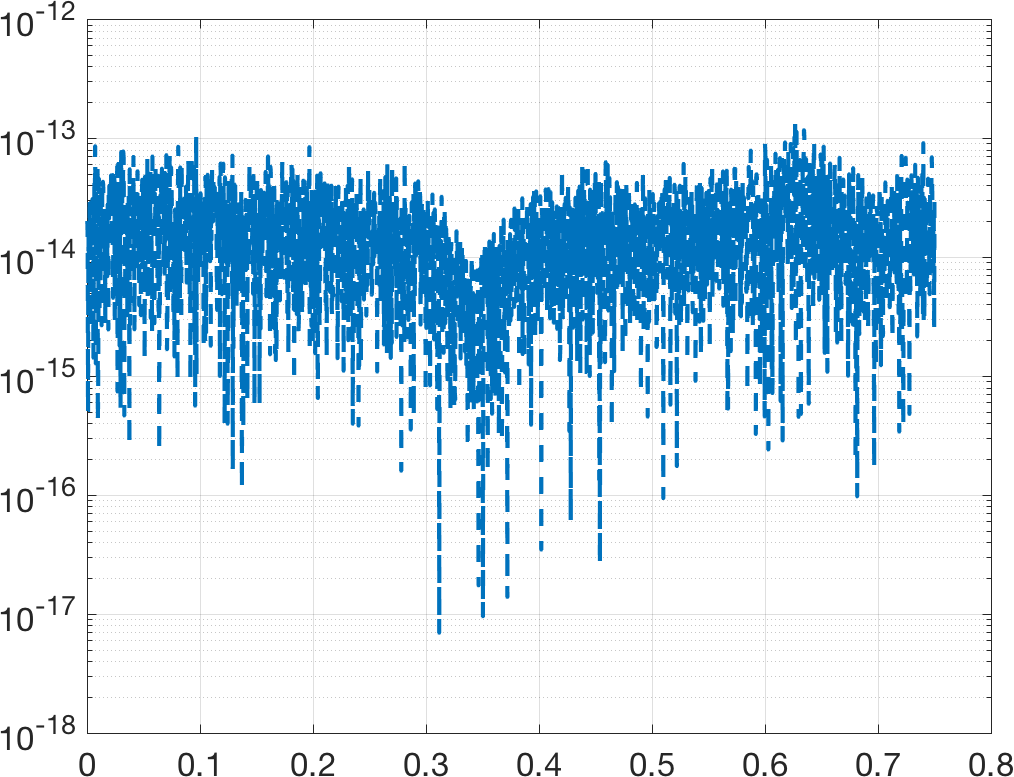}}
\hspace{2em}
\subfloat[Density $\rho$ with 125 modes and no viscosity]{\includegraphics[width=.432\textwidth]{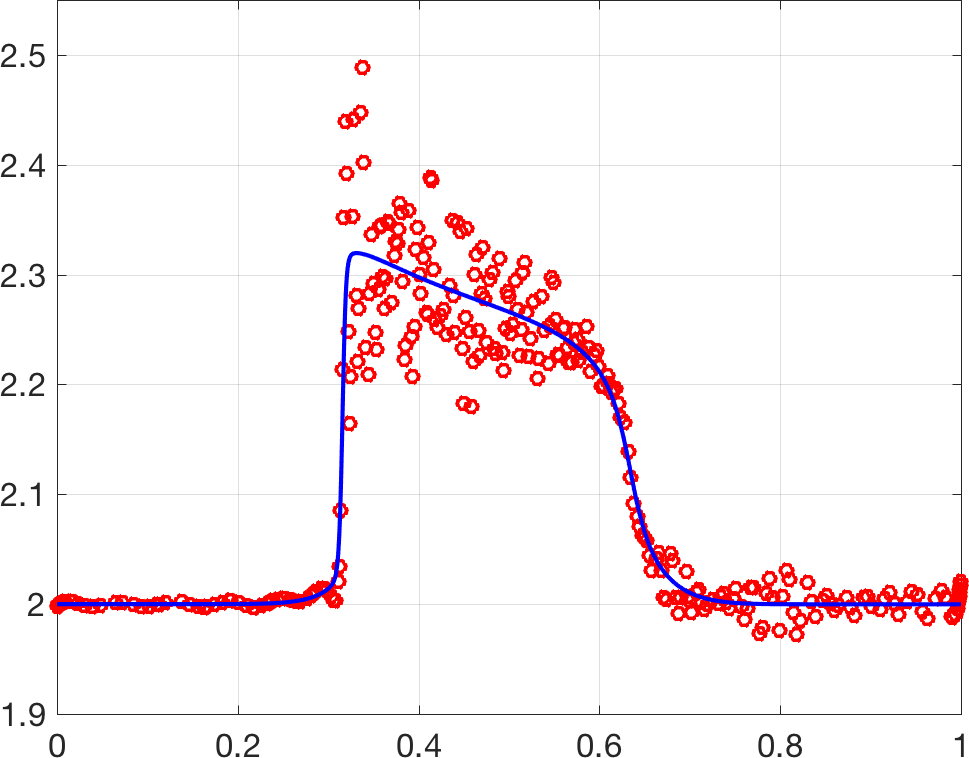}}
\caption{Convective entropy contribution $\LRb{\bm{v}_N^T \bm{V}_h^T\LRp{\bm{Q}_h \circ \bm{F}}\bm{1}}$ over time and reduced order solution with viscosity parameter $\epsilon = 0$ at $T=.75$.  The solution does not blow up despite the presence of large oscillations resulting from the shock.}
\label{fig:novisc}
\end{figure}

Next, we examine the difference in the hyper-reduced treatments of entropy dissipation (\ref{eq:visc1}), (\ref{eq:visc2}), and (\ref{eq:visc3}) described in Section~\ref{sec:diss} (recall that (\ref{eq:visc3}) is not provably entropy dissipative).  If the discrete entropy dissipation $\bm{v}_N^T\bm{d}(\bm{u}_N)$ is positive, then the simulation is entropy stable.  Figure~\ref{fig:entropydiss}  shows the computed entropy dissipation over the time interval $[0,.75]$.  All hyper-reduced treatments produce similar results, with the entropy dissipation produced by the naive approximation (\ref{eq:visc3}) differing most significantly.  The naive approximation differs from the provably entropy stable approximations (\ref{eq:visc1}) and (\ref{eq:visc2}) by about $O\LRp{10^{-1}}$ to $O\LRp{10^{-2}}$, while (\ref{eq:visc1}) differs from (\ref{eq:visc2}) by $O\LRp{10^{-6}}$.  

\begin{figure}[!h]
\centering
\subfloat[Entropy dissipation for 25 modes]{\includegraphics[width=.45\textwidth]{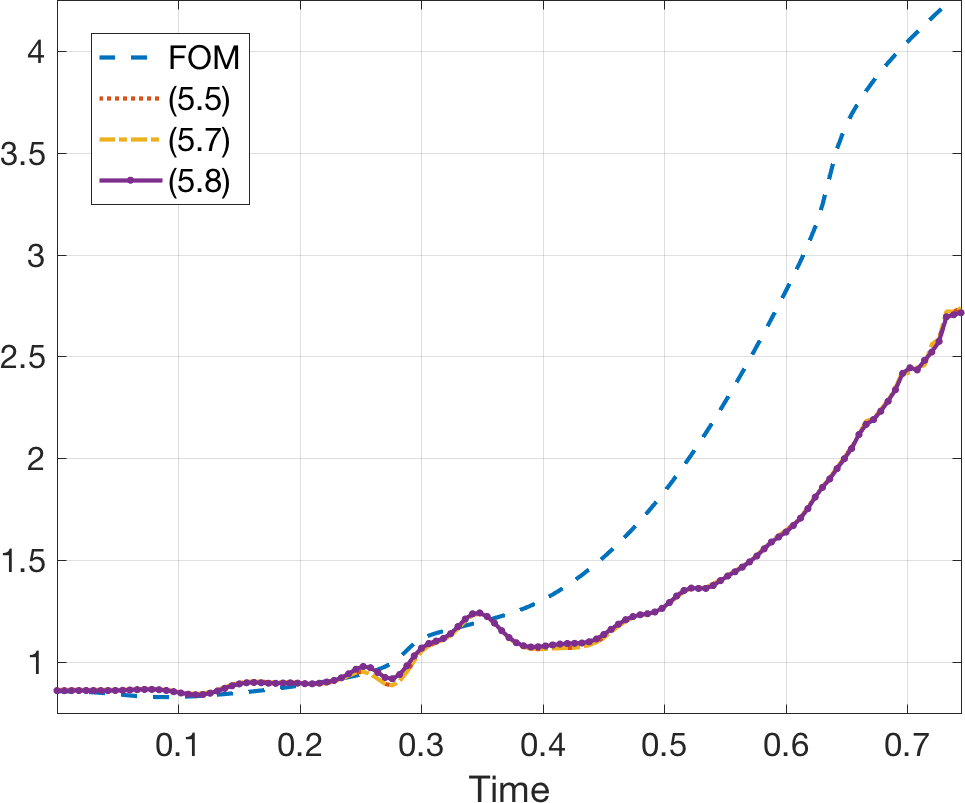}}
\hspace{2em}
\subfloat[Entropy dissipation for 75 modes]{\includegraphics[width=.45\textwidth]{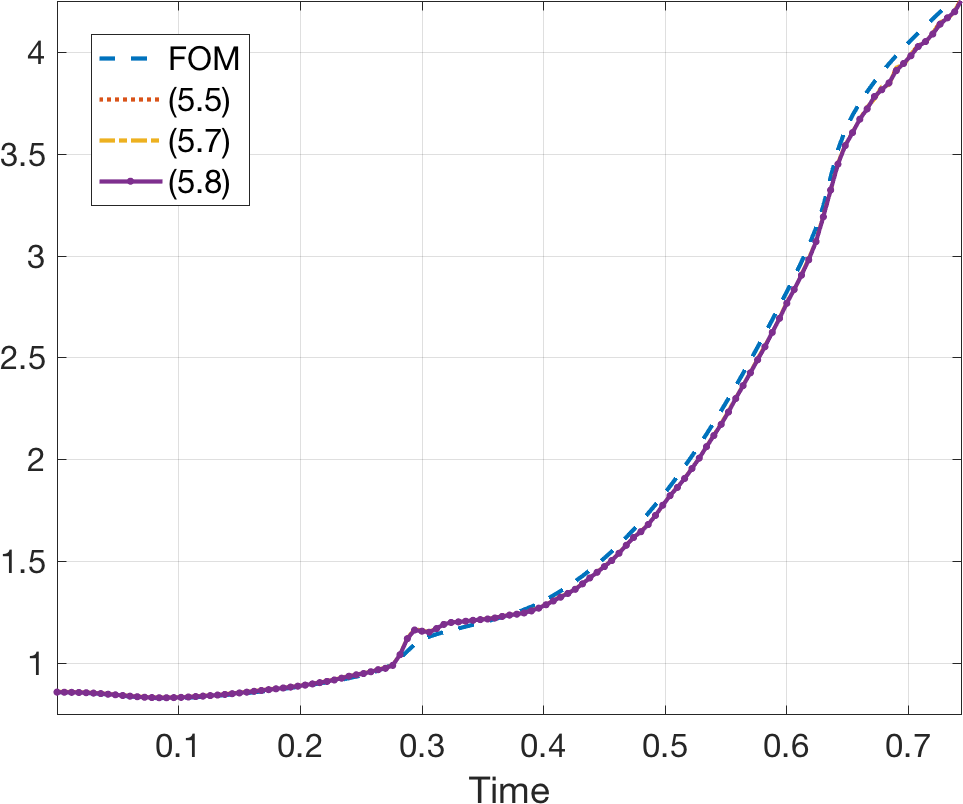}}
\caption{Discrete entropy dissipation $\bm{v}_N^T\bm{d}(\bm{u}_N)$ computed using hyper-reduced treatments of viscous terms (\ref{eq:visc1}), (\ref{eq:visc2}), and (\ref{eq:visc3}) for 25 and 75 modes.}
\label{fig:entropydiss}
\end{figure}

We now compare the evolution of both the discrete entropy \rnote{and the discrete $L^2$ error between all solution components of the full and reduced order models}, and examine the effect of hyper-reduction on the discrete solution in Figure~\ref{fig:entropyevo}.  \rnote{Here, the discrete $L^2$ norm is defined as $\nor{\bm{u}}^2 = \Delta x^2 \bm{u}^T\bm{u}$ in 1D or $\nor{\bm{u}}^2 = (\Delta x \Delta y)^2 \bm{u}^T\bm{u}$ in 2D, and is the analogue of the continuous $L^2$ norm evaluated at discretization points for the full order model.}
Both error and entropy are computed on the original grid of the full order model (e.g., without hyper-reduction), and the error is computed as the discrete relative $L^2$ error over all the conservative variables.  We observe that the evolution of the average discrete entropy for the reduced model rapidly approaches the average discrete entropy for the full order model as the number of modes increases from 25 to 75.  For 125 modes, the average entropies for the ROM and FOM are indistinguishable and are not shown.  The errors behave similarly, decreasing as the number of modes increases.  For reference, we also include errors computed without hyper-reduction using (\ref{eq:esrom}).  We observe that the hyper-reduced errors are virtually identical to errors without hyper-reduction in all cases.

\begin{figure}[!h]
\centering
\subfloat[Error between ROM and FOM]{\includegraphics[height=.282\textheight]{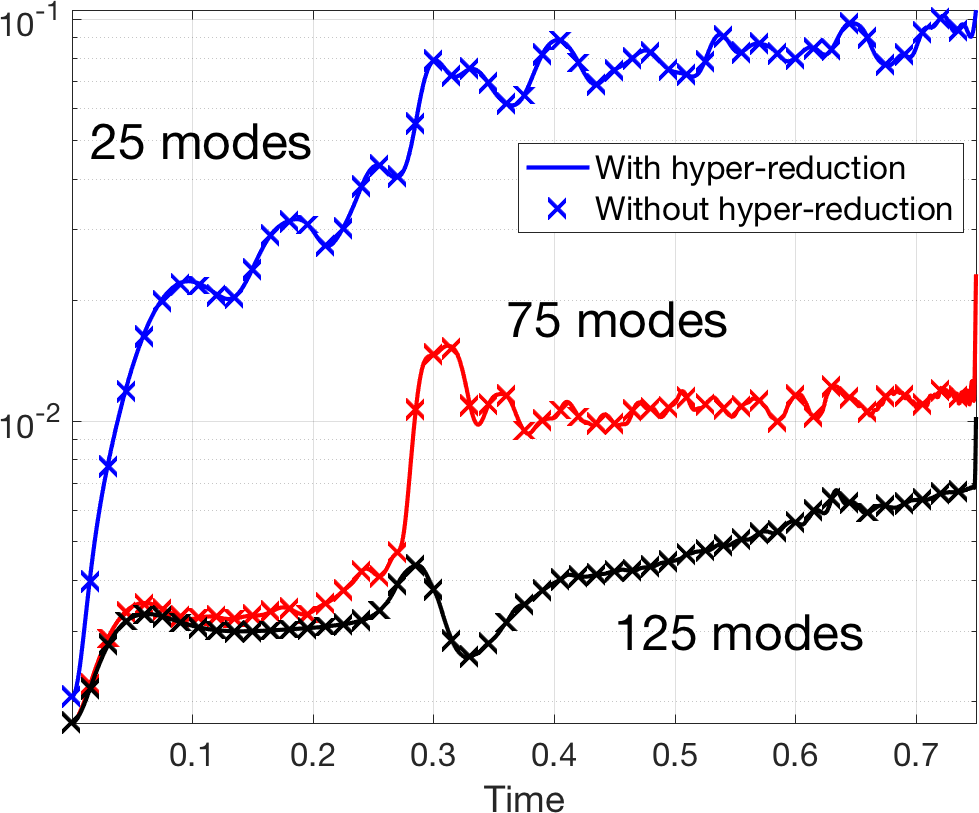}}
\hspace{2em}
\subfloat[Entropy over time]{\includegraphics[height=.285\textheight]{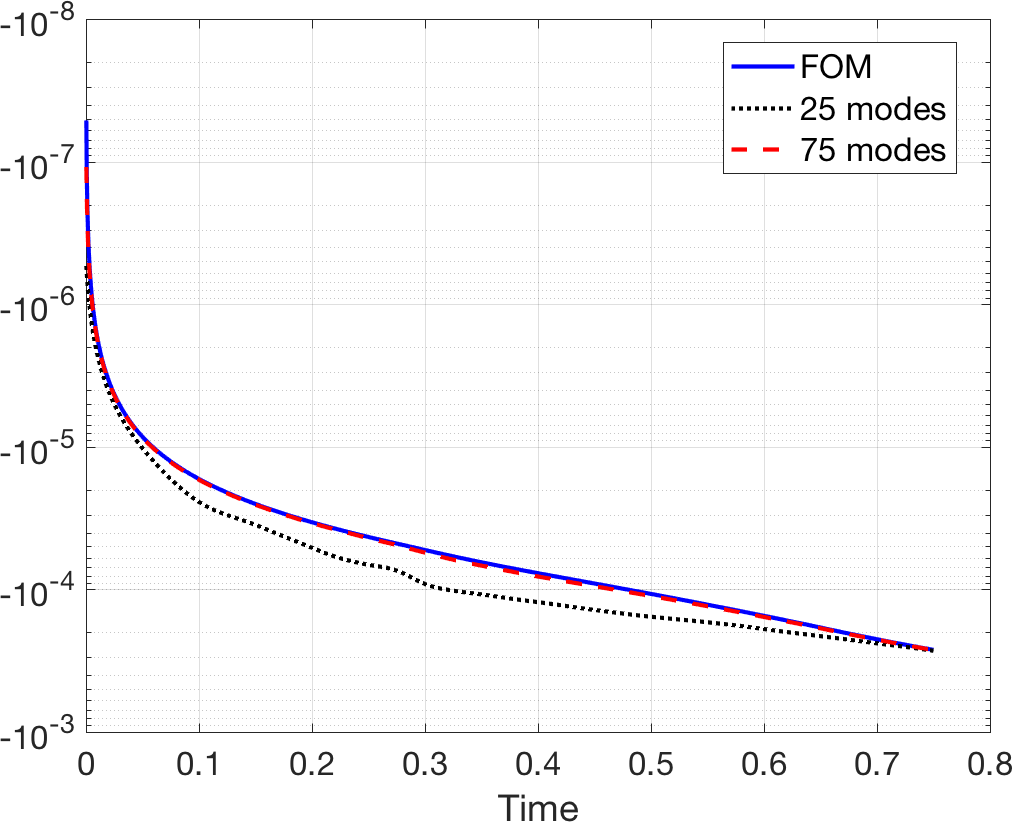}}
\caption{Evolution of error (with and without hyper-reduction) and entropy over time interval $[0,.75]$.}
\label{fig:entropyevo}
\end{figure}





\subsection{2D Euler equations}

We now study the Euler equations in 2D.  We begin by considering periodic boundary conditions on the domain $[-1,1]^2$ and a smoothed version of the Kelvin-Helmholtz instability.  The initial condition is adapted from \cite{munz1989numerical, maboudi2018conservative}
\begin{gather*}
\rho = 1 + \frac{1}{1+e^{-(y+1/2)/\sigma^2}} - \frac{1}{1+e^{-(y-1/2)/\sigma^2}}, \qquad
u =  \frac{1}{1+e^{-(y+1/2)/\sigma^2}} - \frac{1}{1+e^{-(y-1/2)/\sigma^2}} - \frac{1}{2}, \\
v = \alpha\sin(2\pi x) \LRp{e^{-(y+1/2)/\sigma^2} - e^{-(y-1/2)/\sigma^2}},\qquad
p = 2.5.
\end{gather*}
In the following experiments, we set $\alpha = .1$ and $\sigma = .1$.  

\begin{figure}
\centering
\includegraphics[width=.4\textwidth]{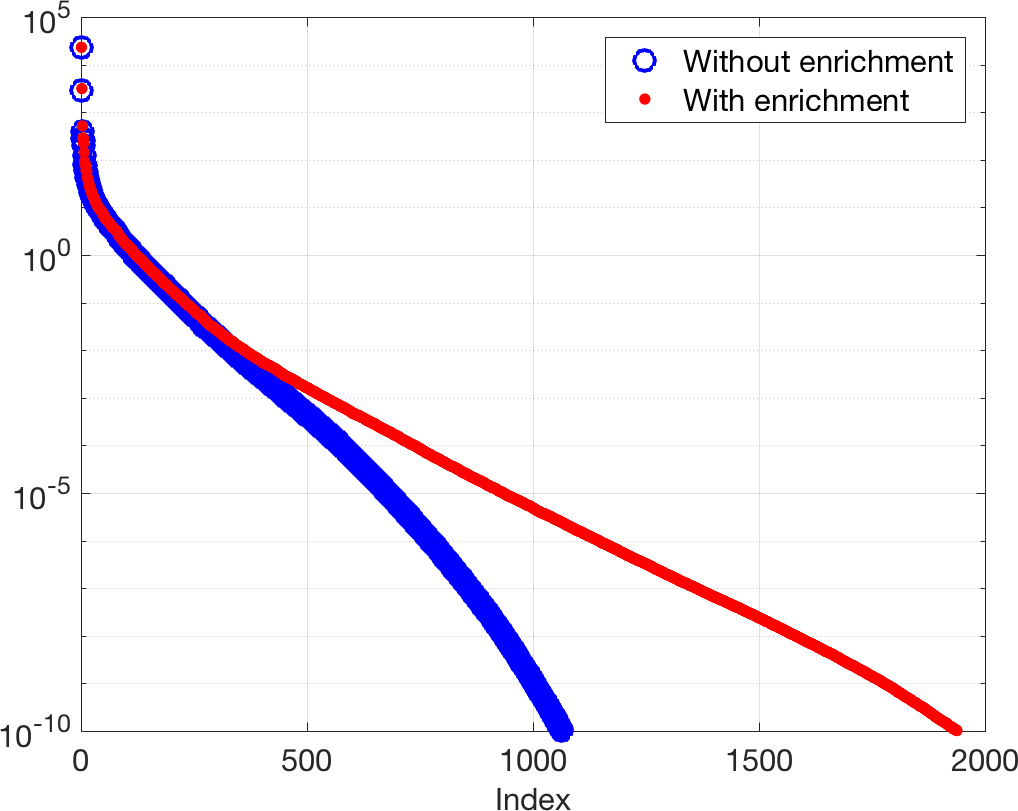}
\caption{Decay of singular values for snapshots of the 2D smoothed Kelvin-Helmholtz instability.}
\label{fig:khsvd}
\end{figure}
We utilize a full order model with $200\times 200$ grid ($40000$ cells), which is run until time $T = 3$.  We generate 301 solution snapshots with which to compute POD modes.  The decay of the singular values of the solution snapshots with and without enrichment by the entropy variables is shown in Figure~\ref{fig:khsvd}.  We observe that, compared to the 1D case, the singular values decay more slowly when enriched with snapshots of the entropy variables.  However, as in 1D, the differences between the singular values with and without entropy variable enrichment are significantly larger for lower energy modes as compared to higher energy modes.  

\begin{figure}
\centering
\subfloat[Full order model]{\includegraphics[width=.32\textwidth]{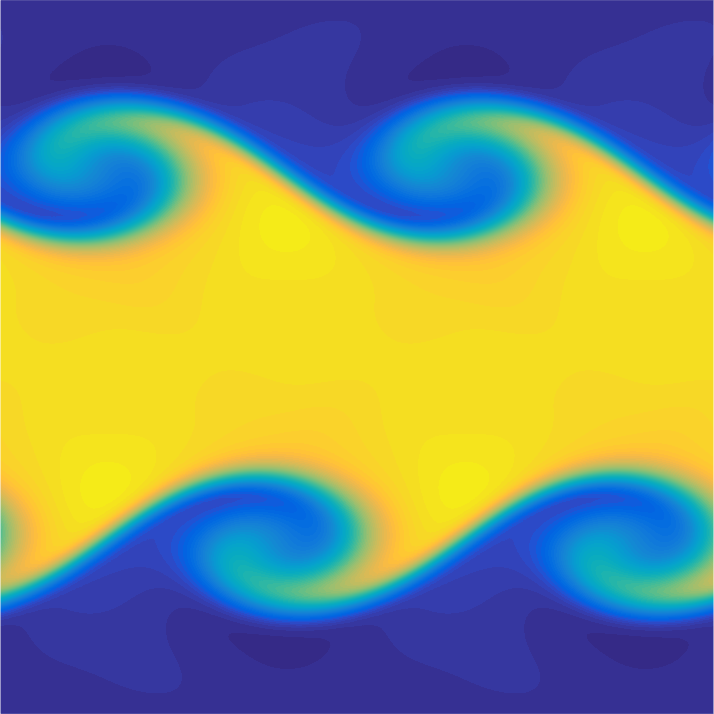}}
\hspace{.175em}
\subfloat[Reduced order model, 75 modes]{\includegraphics[width=.32\textwidth]{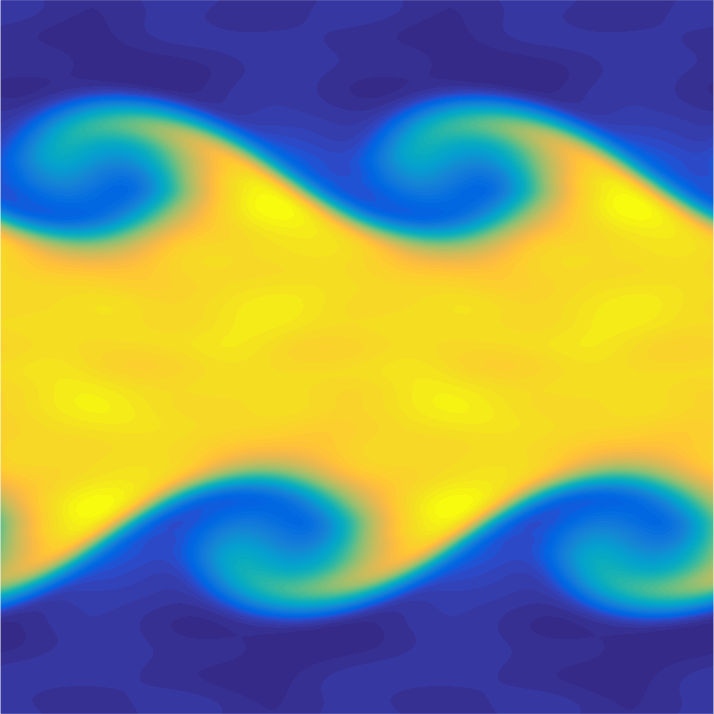}}
\hspace{.1em}
\subfloat[Hyper-reduced points]{\includegraphics[width=.32\textwidth]{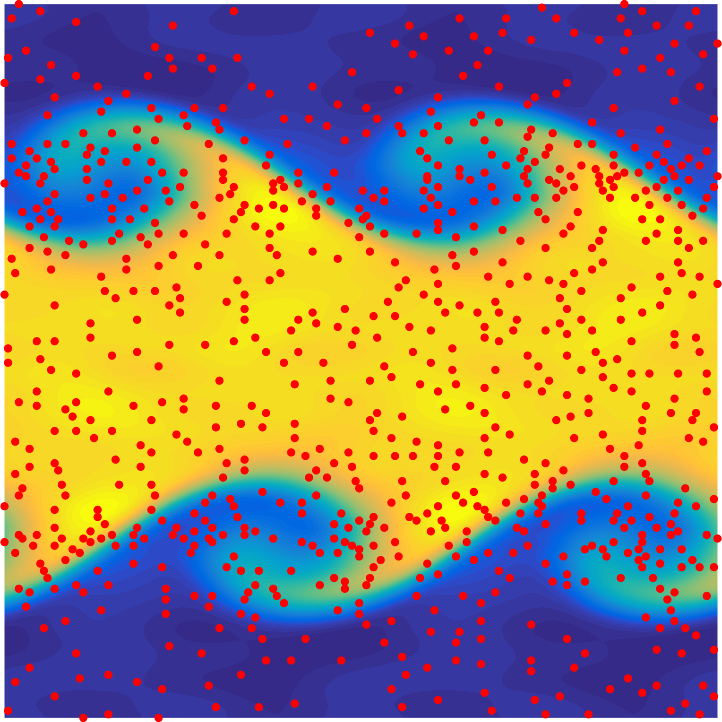}}
\caption{Full and reduced order density for the 2D smoothed Kelvin-Helmholtz instability at $T = 3$ with $\epsilon = 1e-3$.  The reduced order model uses 75 modes and $884$ hyper-reduced points (in \textcolor{red}{red}).  The relative $L^2$ error is $.0102$.}
\label{fig:khrom}
\end{figure}
The reduced order model utilizes $75$ POD modes to achieves a $1.02\%$ relative $L^2$ error.  We note that the difference between the 75th singular value of the snapshot matrix with entropy variable enrichment is only about $6\%$ larger than the 75th singular value without entropy variable enrichment.  This suggests that entropy variable enrichment does not significantly impact the accuracy of the 75-mode POD approximation.  

Since little difference is observed between the hyper-reducted treatments of viscosity, the naive treatment (\ref{eq:visc3}) is used.  No additional stabilizing points were necessary for this setting, as the condition numbers of the hyper-reduced $x$ and $y$ test mass matrices were both $O(1)$.  Figure~\ref{fig:khrom} compares results for the full and reduced order models, and shows the hyper-reduced points selected by the greedy empirical cubature algorithm.  

We next consider a 2D domain with wall boundary conditions.  Entropy stable reflective wall boundary conditions are again imposed using ``mirror states''.   Let $\bm{n}$ denote the outward normal at a wall, and let $\bm{u}_{\bm{n}} = un_x + vn_y$ and $\bm{u}_{\bm{n}^\perp} = un_y - vn_x$ denote the normal and tangential components of the velocity at the wall.  We use a boundary numerical flux $\bm{f}^* = \bm{f}_S\LRp{\bm{u}^+,\bm{u}}$ augmented with local Lax-Friedrichs penalization, where the exterior state $\bm{u}^+$ is defined by
\[
\rho^+ = \rho, \qquad \bm{u}_{\bm{n}}^+ = -\bm{u}_{\bm{n}}, \qquad \bm{u}_{\bm{n}^\perp}^+ = \bm{u}_{\bm{n}^\perp}, \qquad p^+ = p.  
\]
We set a Gaussian pulse initial condition 
\[
\rho = 1 + e^{-50\LRp{x^2+(y+1/2)^2}}, \qquad \bm{u} = \bm{0}, \qquad p = \rho^{\gamma}.  
\]
The full order model is computed on a $150\times 150$ grid with viscosity coefficient $\epsilon = 1e-3$.  The POD basis is computed from 75 solution snapshots with a CFL of $.5$.  Figure~\ref{fig:pulse2dsvd} shows the decay of the singular values with and without entropy variable enrichment.  As before, we observe that the singular values decay more slowly under entropy variable enrichment, but that the difference in the singular values with and without enrichment is more evident for low energy modes.  The 25th singular value with entropy variable enrichment is less than $5\%$ larger than the 25th singular value without entropy variable enrichment, suggesting that entropy variable enrichment does not significantly affect the accuracy of the 25-mode POD approximation.  

\begin{figure}
\centering
\includegraphics[width=.4\textwidth]{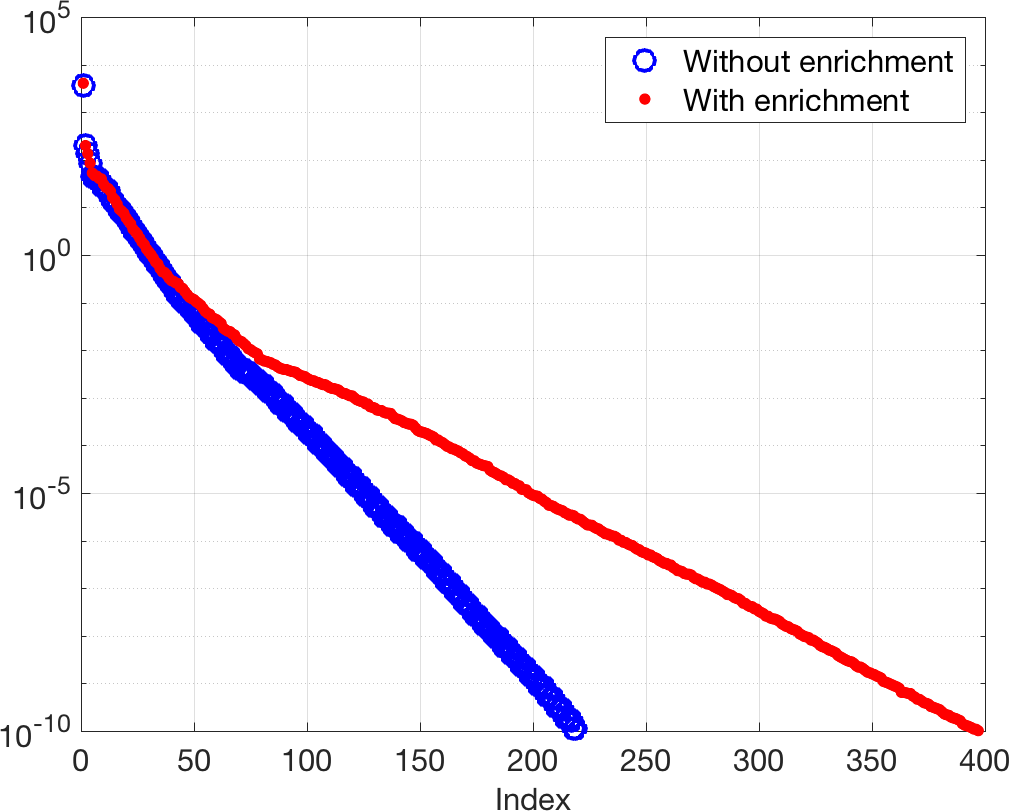}
\caption{Decay of singular values for snapshots of the 2D Gaussian pulse problem.}
\label{fig:pulse2dsvd}
\end{figure}

A 25-mode reduced order solution is shown in Figures~\ref{fig:pulse2d}.  The hyper-reduction algorithm adds one single stabilizing point to reduce the condition number of the $x$-coordinate test mass matrix from $O(10^{8})$ to approximately $5.71$.  The solution does not form shock discontinuities, and is chosen instead to test the imposition of boundary conditions.  We note that, due to the nature of the linear programming software used to compute the boundary weights, the constraints (\ref{eq:sbpconstraints}) on the boundary weights $\bm{w}_b$ are imposed to a tolerance of $5e-8$.  Despite this inexact enforcement of constraints, the computed entropy RHS remained small, oscillating around $10^{-11}$ in the absence of entropy dissipative terms.  

\begin{figure}
\centering
\subfloat[Full order model]{\includegraphics[width=.32\textwidth]{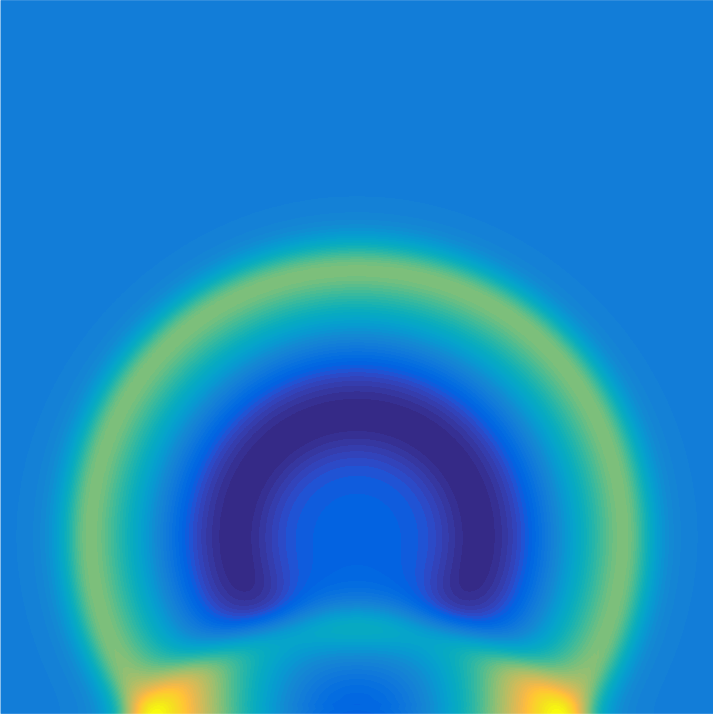}}
\hspace{.1em}
\subfloat[Reduced order model]{\includegraphics[width=.32\textwidth]{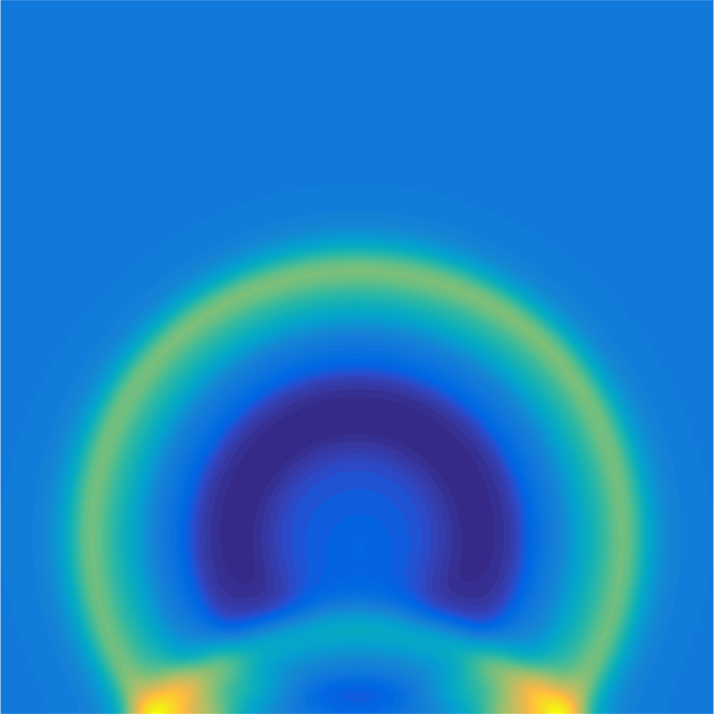}}
\hspace{.1em}
\subfloat[Hyper-reduction points]{\includegraphics[width=.32\textwidth]{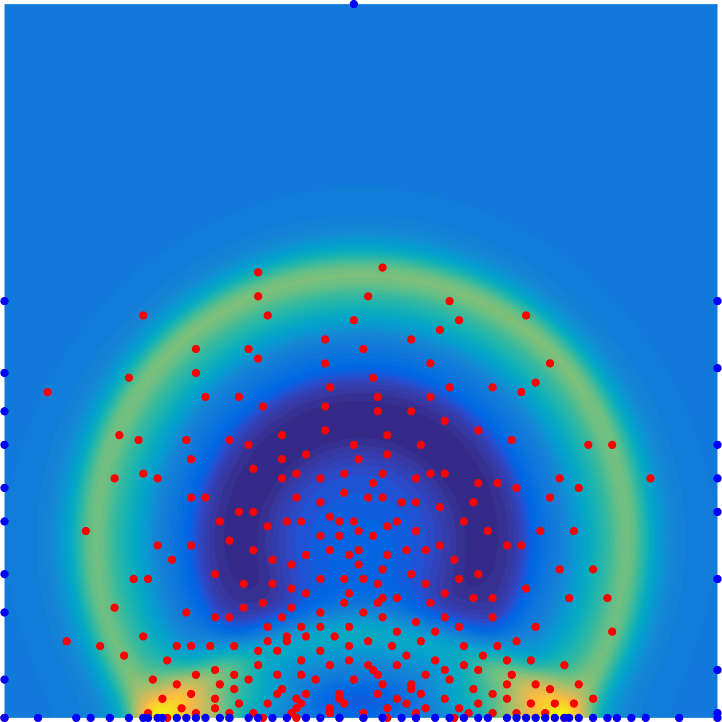}}
\caption{Full and reduced order density at time $T = .25$ computed with $\epsilon = 1e-3$.  The reduced order model uses 25 modes, 300 hyper-reduced points (in \textcolor{red}{red}), and 66 hyper-reduced boundary points (in \textcolor{blue}{blue}).  The relative $L^2$ error is $.0071$.}
\label{fig:pulse2d}
\end{figure}

\subsubsection{2D Riemann problem}

We finally consider a 2D Riemann problem to examine the stability of entropy stable ROMs for under-resolved shock solutions. Because typical ``natural'' boundary conditions are not entropy stable \cite{chen2017entropy}, we modify the problem to use periodic boundary conditions \cite{chan2017discretely} as shown in Figure~\ref{fig:riemann}.  The FOM is run until final time $T=.25$ on a $200\times 200$ grid with a CFL of $.25$ and viscosity coefficient $\epsilon = 5e-3$.  The initial conditions are taken from \cite{lax1998solution, kurganov2002solution} and are smoothed by applying a 3-point average 5 times in each coordinate direction.  In contrast to the Kelvin-Helmholtz instability and pulse problems, the singular values of the entropy variable enriched snapshots decay noticeably more slowly than the singular values of the non-enriched solution snapshots. Figure~\ref{fig:riemann2dsvd} shows the decay of the snapshot singular values along with a zoomed in view. It can be observed that the enriched and non-enriched snapshot singular values begin to differ after the first 20 terms. 

\begin{figure}
\centering
\subfloat[Singular values]{\includegraphics[height=.24\textheight]{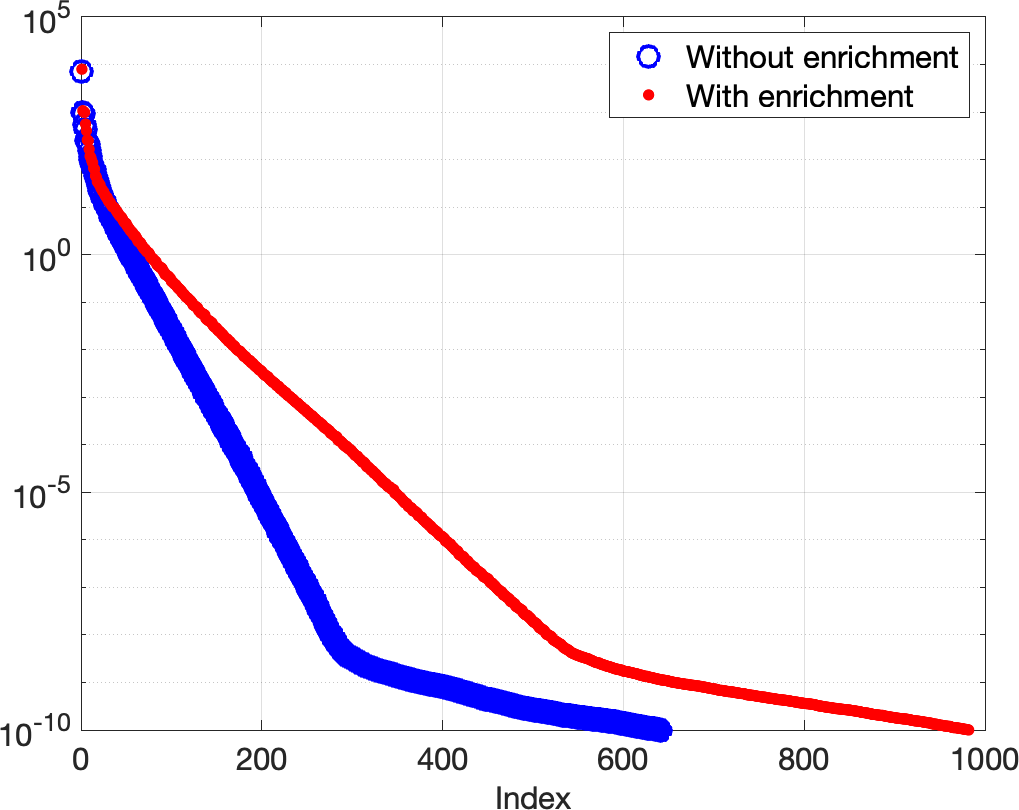}}
\hspace{1.5em}
\subfloat[Zoomed view]{\includegraphics[height=.24\textheight]{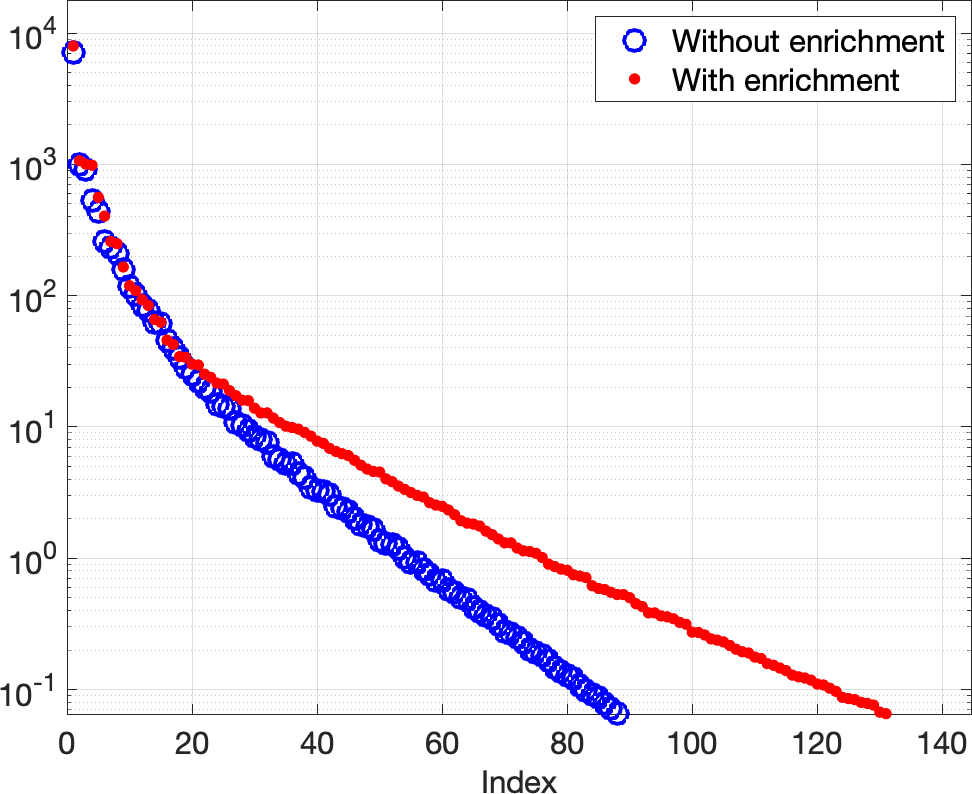}}
\caption{Decay of singular values for snapshots for the 2D Riemann problem.}
\label{fig:riemann2dsvd}
\end{figure}

The ROM uses $50$ modes and $812$ hyper-reduced points, and achieves a final $L^2$ error of $0.03278$.  Despite being highly under-resolved with oscillations trailing the shock, the ROM runs stably. These oscillations are the result of approximating a traveling shock using a linear POD subspace, and are not signs of instability within the numerical scheme. The solution appears reasonably well-approximated outside of the domain over which the shock propagates. 

\begin{figure}
\centering
\subfloat[FOM results]{\includegraphics[width=.32\textwidth]{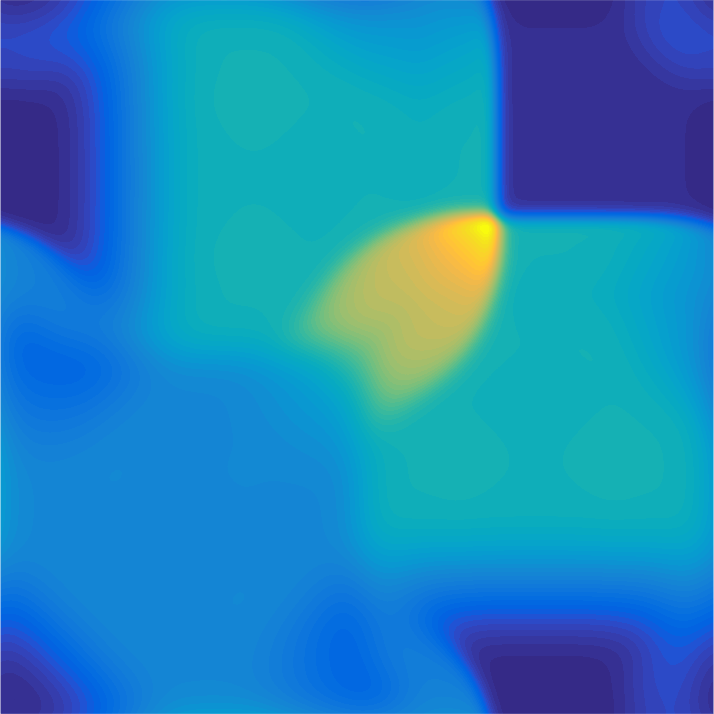}}
\hspace{.25em}
\subfloat[ROM results]{\includegraphics[width=.32\textwidth]{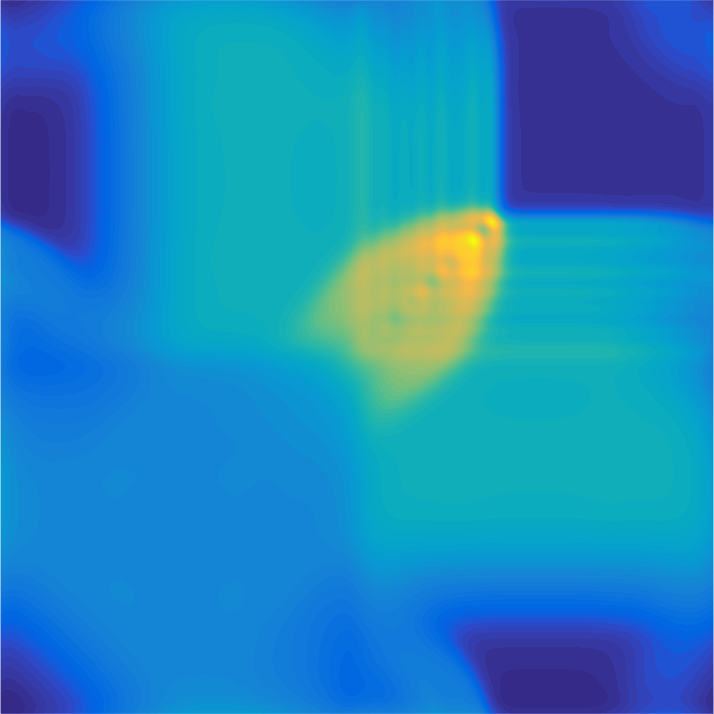}}
\hspace{.25em}
\subfloat[Hyper-reduced points]{\includegraphics[width=.32\textwidth]{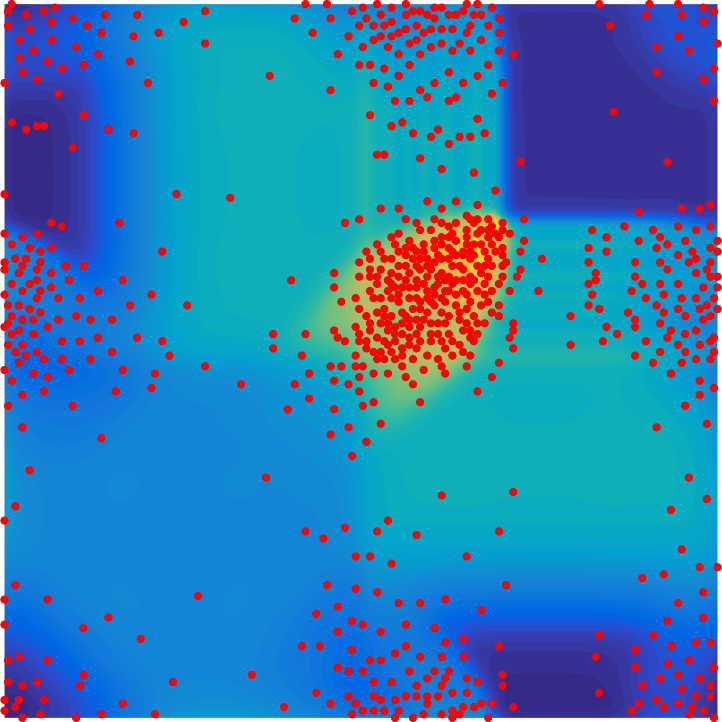}}
\caption{Comparison of ROM and FOM behavior for the 2D Riemann problem. The reduced order model uses 50 modes and 812 hyper-reduced points (in \textcolor{red}{red}).  The relative $L^2$ error is $.03278$.}
\label{fig:riemann}
\end{figure}

\subsection{\bnote{On computational performance}}

\bnote{We note that the entropy stable reduced order models implemented in this work do not reduce computational cost compared to the original full order models. This is due to explicit time-stepping, which is used in this paper to validate the proposed semi-discrete formulation.  However, we expect that for implicit time-stepping, entropy stable ROMs will see more significant efficiency gains.  

The cost of explicit time-stepping schemes scales with the cost of the ODE right hand side evaluation, and the cost of a right hand side evaluation for an entropy stable scheme scales with the number of entropy conservative flux evaluations required.  The number of entropy conservative flux evaluations is the same as the number of nonzero entries in the matrix $\bm{Q}^i$, since the nonlinear term is $\LRp{\bm{Q}^i\circ\bm{F}^i}\bm{1}$ and the entries of the matrix $\bm{F}^i$ (which are flux evaluations between different states) are evaluated on the fly.  

For the 2D full order methods in this paper, the differentiation matrices $\bm{Q}^i$ are sparse with $O(K^2)$ non-zero entries.  Thus, the cost of a 2D entropy stable finite volume method is $O(K^2)$ nonlinear flux evaluations for each coordinate direction.  For a reduced order model, however, the hyper-reduced matrices $\bm{Q}^i_t$ are dense, and evaluating $(\bm{Q}^i_t \circ \bm{F}^i)\bm{1}$  requires $O(N_s^2)$ flux evaluations where $N_s$ is the number of hyper-reduced points.  As an example, consider the Kelvin-Helmholtz results in Figure~\ref{fig:khrom}.  The full order model requires roughly $40000$ flux evaluations.  However, because the ROM has $884$ hyper-reduced points, roughly $781456$ flux evaluations are required. The result is a ROM which is more expensive per time-step than the original full order model.  

While the proposed ROMs could still provide savings for sufficiently large full order models, this increased cost is a significant issue for explicit time-stepping.  The cost is offset slightly if the ROM has a larger CFL \cite{marley2015reduced}; however, this situation only occurs if a smaller number of modes are used to simulate the solution.  The situation changes for implicit time-stepping, as full order models require multiple solutions of large (but sparse) linear systems at each time-step.  For a $K\times K$ grid, these linear systems would require $O(K^2)$ flux evaluations to assemble, but we expect the solution of the linear system to be the dominant cost.  For a ROM using implicit time-stepping, the assembly of the linear system still requires $O(N_s^2)$ flux evaluations, but the size of the system scales with the number of modes, which is notably smaller than even the number of hyper-reduced points.  For example, the Kelvin-Helmholtz example uses only $75$ basis functions per component, compared to $884$ hyper-reduced points.  We thus expect that the use of entropy stable ROMs will result in a more significant speedup under implicit time-stepping, especially when combined with efficient methods for computing Jacobian matrices \cite{chan2020explicit}.  }

\section{Conclusion}  We have presented a methodology for constructing projection-based reduced order models for nonlinear conservation laws by combining a reduced basis, a modified Galerkin projection, and tailored hyper-reduction techniques.  The main novelty of these new reduced models is the approximation of the nonlinear convective term, which combines a ``flux differencing'' approach with an appropriate hyper-reduced approximation of the differentiation matrix.

Future work will aim to address computational costs associated with entropy stable reduced models.  Unlike standard hyper-reduction techniques, the number of nonlinear evaluations necessary scales with $O(N_s^2)$ rather than $O(N_s)$, where $N_s$ is the number of hyper-reduced sampling points.  This discrepancy is due to the fact that the hyper-reduction presented here approximates a nonlinear matrix, rather than a nonlinear vector.  These costs can be reduced by combining domain decomposition \cite{lucia2003reduced} and a discontinuous Galerkin-type discretization \cite{chan2017discretely} to produce multi-domain reduced order models.  Suppose there are $k = 1,\ldots, K$ subdomains.  Then, a multi-domain reduced order model is expected to decrease costs if the sum of the squares of the number of subdomain sampling points $\sum_k (N_s^k)^2$ is significantly smaller than the global number of sampling points $N_s^2$.  Future work will also investigate implicit time-stepping for entropy stable reduced models, as additional costs associated with the spatial formulation may be offset by the reduction in cost for solving smaller matrix systems during the linearization process.

\section{Acknowledgments}

Jesse Chan gratefully acknowledges support from the National Science Foundation under awards DMS-1719818, DMS-1712639, and DMS-CAREER-1943186. Jesse Chan also thanks Matthias Heinkenschloss, Masayuki Yano, Matthew Zahr, and Irina Tezuar for informative discussions.

\bibliographystyle{unsrt}
\bibliography{refs2}

\end{document}